\documentclass[11pt]{article}
\usepackage{latexsym,amsfonts,amssymb,amsmath,amsthm}
\usepackage{graphicx}

\usepackage[usenames,dvipsnames]{color}
\usepackage{ulem}

\begin{document}
\setlength{\baselineskip}{16pt}

\parindent 0.5cm
\evensidemargin 0cm \oddsidemargin 0cm \topmargin 0cm \textheight
22cm \textwidth 15cm \footskip 2cm \headsep 0cm

\newtheorem{theorem}{Theorem}[section]
\newtheorem{lemma}{Lemma}[section]
\newtheorem{proposition}{Proposition}[section]
\newtheorem{definition}{Definition}[section]
\newtheorem{example}{Example}[section]
\newtheorem{corollary}{Corollary}[section]

\newtheorem{remark}{Remark}[section]

\numberwithin{equation}{section}

\def\p{\partial}
\def\I{\textit}
\def\R{\mathbb R}
\def\C{\mathbb C}
\def\u{\underline}
\def\l{\lambda}
\def\a{\alpha}
\def\O{\Omega}
\def\e{\epsilon}
\def\ls{\lambda^*}
\def\D{\displaystyle}
\def\wyx{ \frac{w(y,t)}{w(x,t)}}
\def\imp{\Rightarrow}
\def\tE{\tilde E}
\def\tX{\tilde X}
\def\tH{\tilde H}
\def\tu{\tilde u}
\def\d{\mathcal D}
\def\aa{\mathcal A}
\def\DH{\mathcal D(\tH)}
\def\bE{\bar E}
\def\bH{\bar H}
\def\M{\mathcal M}
\renewcommand{\labelenumi}{(\arabic{enumi})}

\def\disp{\displaystyle}
\def\undertex#1{$\underline{\hbox{#1}}$}
\def\card{\mathop{\hbox{card}}}
\def\sgn{\mathop{\hbox{sgn}}}
\def\exp{\mathop{\hbox{exp}}}
\def\OFP{(\Omega,{\cal F},\PP)}
\newcommand\JM{Mierczy\'nski}
\newcommand\RR{\ensuremath{\mathbb{R}}}
\newcommand\CC{\ensuremath{\mathbb{C}}}
\newcommand\QQ{\ensuremath{\mathbb{Q}}}
\newcommand\ZZ{\ensuremath{\mathbb{Z}}}
\newcommand\NN{\ensuremath{\mathbb{N}}}
\newcommand\PP{\ensuremath{\mathbb{P}}}
\newcommand\abs[1]{\ensuremath{\lvert#1\rvert}}

\newcommand\normf[1]{\ensuremath{\lVert#1\rVert_{f}}}
\newcommand\normfRb[1]{\ensuremath{\lVert#1\rVert_{f,R_b}}}
\newcommand\normfRbone[1]{\ensuremath{\lVert#1\rVert_{f, R_{b_1}}}}
\newcommand\normfRbtwo[1]{\ensuremath{\lVert#1\rVert_{f,R_{b_2}}}}
\newcommand\normtwo[1]{\ensuremath{\lVert#1\rVert_{2}}}
\newcommand\norminfty[1]{\ensuremath{\lVert#1\rVert_{\infty}}}

\title{Nonlocal dispersal equations with almost periodic dependence. I. Principal spectral theory}

\author{Maria Amarakristi Onyido\,\, and\,\,
 Wenxian Shen\\
Department of Mathematics and Statistics\\
Auburn University\\
Auburn University, AL 36849 }

\date{}
\maketitle

\noindent {\bf Abstract.}
This series of two papers is devoted to the study of the principal spectral theory of
nonlocal dispersal operators with almost periodic dependence and the study of the asymptotic
dynamics of nonlinear nonlocal dispersal equations with almost periodic dependence. In this first
part of the series, we investigate the principal spectral theory of nonlocal dispersal operators
from two aspects: top Lyapunov exponents and generalized principal eigenvalues. Among others,
we provide various characterizations of the top Lyapunov exponents and generalized principal
eigenvalues, establish the relations between them, and study the effect of time and space
 variations on them. In the second part of the series, we will study the asymptotic dynamics of
nonlinear nonlocal dispersal equations with almost periodic dependence applying the principal spectral theory to be developed in
this part.

\bigskip


\bigskip


\newpage

\section{Introduction}
\setcounter{equation}{0}

 This  series of two papers is devoted to the study of the principal spectral theory of the following linear nonlocal dispersal equation,
\begin{equation}
\label{main-linear-eq}
\p_t u=\int_D \kappa(y-x)u(t,y)dy+a(t,x)u,\quad x\in \bar D,
\end{equation}
and to the study of the asymptotic dynamics of the following nonlinear nonlocal dispersal equation,
\begin{equation}
\label{main-nonlinear-eq}
\p_t u=\int_D \kappa(y-x)u(t,y)dy+u f(t,x,u),\quad x\in\bar D,
\end{equation}
where $D\subset \RR^N$ is a bounded domain or $D=\RR^N$, and $\kappa(\cdot)$, $a(\cdot,\cdot)$ and $f(\cdot,\cdot,\cdot)$ satisfy

 \medskip

 \noindent {\bf (H1)} $\kappa(\cdot)\in C^1(\RR^{N},[0,\infty))$, $\kappa(0)>0$,
  $\int_{\RR^N}\kappa(x)dx=1$, and there are $\mu,M>0$ such that $\kappa(x)\le e^{-\mu|x|}$ and $|\nabla \kappa|\le e^{-\mu |x|}$
  for $|x|\ge M$.

 \medskip

 \noindent {\bf (H2)}  $a(t,x)$ is uniformly continuous in $(t,x)\in\RR\times \bar D$, and is almost periodic in $t$
 uniformly with respect to $x\in\bar D$ (see Definition \ref{almost-periodic-def} for the definition of almost periodic functions).

 \medskip

 \noindent {\bf (H3)}  $f(t,x,u)$ is $C^1$ in $u$; $f(t,x,u)$ and $f_u(t,x,u)$
are uniformly continuous in $(t,x,u)\in \RR\times \bar D\times E$ for any bounded set $E\subset \RR$; $f(t,x,u)$ is almost periodic in $t$ uniformly with respect to $x\in\bar D$ and $u$ in bounded sets of $\RR$; $f(t,x,u)<0$ for $u\gg 1$ and any $(t,x)\in\RR\times\bar D$;
$f_u(t,x,u)<0$ for $(t,x,u)\in \RR\times \bar D\times [0,\infty)$.

\medskip

We will  establish the principal spectral theory of \eqref{main-linear-eq} in this  part
 and will study the asymptotic dynamics
of \eqref{main-nonlinear-eq} in the second part.

Recently there has been extensive investigation on the dynamics of populations having a long range dispersal strategy (see \cite{BaLi0, BaZh, BaLi,  BeCoVo1, Co1, Co3, CoDaMa, DeShZh, GaRo1, KaLoSh, LiSuWa, LiWaZh, LiZh, RaShZh,  ShZh1, ShZh2, ShVo, ZhZh1, ZhZh2}, etc.).  The following nonlocal reaction diffusion equations are commonly used models to integrate the  long range dispersal
 for such populations (see \cite{Fif, GrHiHuMiVi,  HuMaMiVi, LuPaLe, Tur}, etc):
\begin{equation}
\label{dirichlet-kpp-eq} \p_t u=\int_\Omega
\kappa(y-x)u(t,y)dy-u(t,x)+ug(t,x,u),\quad x\in\bar \Omega,
\end{equation}
\begin{equation}
\label{neumann-kpp-eq}
 \p_t u=\int_\Omega \kappa(y-x)(u(t,y)-u(t,x))dy+ug(t,x,u),\quad x\in\bar \Omega,
\end{equation}
where $\Omega\subset\RR^N$ is a bounded domain,
and
\begin{equation}
\label{periodic-kpp-eq}
 \p_t u=\int_{\RR^N}
\kappa(y-x)u(t,y)dy-u(t,x)+ug(t,x,u),\quad x\in\RR^N.
\end{equation}
Typical examples of the kernel function $\kappa(\cdot)$ satisfying {\bf (H1)} include the probability density function of the normal distribution  $\kappa(x)=\frac{1}{\sqrt{(2\pi)^N}}e^{-\frac{|x|^2}{2}}$ and any $C^1$ convolution kernel function supported on a bounded ball $B(0,r)=\{x\in\RR^N\,|\, |x|<r\}$.

Note that \eqref{dirichlet-kpp-eq} is the nonlocal dispersal counterpart of the following reaction diffusion equation with Dirichlet boundary condition,
\begin{equation}
\label{random-dirichlet-eq}
\begin{cases}
u_t=\Delta u+u g(t,x,u),\quad x\in \Omega\cr
u=0,\quad x\in \p \Omega,
\end{cases}
\end{equation}
and \eqref{neumann-kpp-eq} is the nonlocal dispersal counterpart of the following reaction diffusion equation with Neumann boundary condition,
\begin{equation}
\label{random-neumann-eq}
\begin{cases}
u_t=\Delta u+u g(t,x,u),\quad x\in \Omega\cr
\frac{\p u}{\p n}=0,\quad x\in \p \Omega.
\end{cases}
\end{equation}
See \cite{CoElRo1, CoElRoWo, ShXi2}  for the relation between \eqref{dirichlet-kpp-eq} and\eqref{random-dirichlet-eq}, and the relation
between \eqref{neumann-kpp-eq} and \eqref{random-neumann-eq}.
Equations \eqref{dirichlet-kpp-eq} and  \eqref{neumann-kpp-eq} can  be viewed as nonlocal dispersal models for populations
 with growth function $ug(t,x,u)$ and with Dirichlet and Neumann boundary conditions, respectively.

Observe that \eqref{dirichlet-kpp-eq} (respectively  \eqref{neumann-kpp-eq}, \eqref{periodic-kpp-eq}) can be written as \eqref{main-nonlinear-eq} with $D=\Omega$ and $f(t,x,u)=-1+g(t,x,u)$ (respectively  $D=\Omega$ and $f(t,x,u)=-\int_D \kappa(y-x)dy+g(t,x,u)$,
$D=\RR^N$ and $f(t,x,u)=-1+g(t,x,u)$).
Hence the theory on the asymptotic dynamics of \eqref{main-nonlinear-eq} to be developed in the second part of the series can be applied to  \eqref{dirichlet-kpp-eq}, \eqref{neumann-kpp-eq}, and \eqref{periodic-kpp-eq}.
 Observe also that $u(t,x)\equiv 0$ is a solution of \eqref{main-nonlinear-eq}, which is refereed to as the {\it trivial solution}
of \eqref{main-nonlinear-eq}. If $a(t,x)=f(t,x,0)$, then \eqref{main-linear-eq}
is the linearization of \eqref{main-nonlinear-eq} at this trivial solution. Hence the principal spectral theory to be established for \eqref{main-linear-eq} in this part of the series
 has its own interests and also plays an important role in the study of the asymptotic dynamics of \eqref{main-nonlinear-eq}.

 Principal spectrum for linear random dispersal or reaction diffusion  equations has been extensively studied and is quite well understood in many cases. For example,
 consider the following random dispersal counterpart of \eqref{main-linear-eq} on a bounded smooth domain $D$ with Dirichlet boundary condition,
\begin{equation}
\label{random-linear-eq}
\begin{cases}
u_t=\Delta u+a(t,x)u,\quad x\in D\cr
u=0\quad x\in \p D.
\end{cases}
\end{equation}
For the periodic case ($a(t+T,x)=a(t,x)$ for all $x\in D$ and $t\in\RR$), there is well-known theory
 (see \cite{Hes}) yielding the existence of a principal eigenvalue $\lambda(a)$ and  eigenfunction $\phi(t,x)$, that is,
$$
\begin{cases}
-\phi_t(t,x)+\Delta \phi(t,x)+a(t,x)\phi(t,x)=\lambda(a)\phi(t,x),\quad x\in D\cr
\phi(t,x)=0\quad x\in \p D\cr
\phi(t+T,x)=\phi(t,x)>0\quad \forall\, t\in\RR,\,\, x\in D.
\end{cases}
$$
 Note that
 the principal eigenvalue of \eqref{random-linear-eq} in the time periodic case
 is a notion
related to the existence of an eigen-pair: an eigenvalue associated with a positive eigenfunction.
The principal eigenvalue theory for \eqref{random-linear-eq} in the time periodic case has been well extended to general time dependent  case  with the principal eigenvalue and eigenfunction in the time periodic case being replaced by principal Lyapunov exponents and principal Floquet bundles,  respectively
(see \cite{HuPo, HuPoSa, MiSh1, MiSh2, ShYi}, etc.).

 The principal spectrum for various special cases of \eqref{main-linear-eq} has been studied by many authors.  For example, when $D$ is bounded and $a(t,x)$ is independent of $t$ or periodic in $t$, the principal spectrum of \eqref{main-linear-eq} has been studied in \cite{Co1, GaRo,   HeShZh, HuShVi, LiCoWa,  RaSh, ShVi, ShXi, ShXi1, ShZh1, ShVo, SuLiLoYa}, etc.. When $D=\RR^N$ and $a(t,x)$ is periodic in both $t$ and $x$, or $a(t,x)\equiv a(x)$,  the principal spectrum of \eqref{main-linear-eq} has been studied in \cite{BeCoVo, CoDaMa1, RaSh, ShZh1}, etc..
  In comparison with the random dispersal operators, even when $a(t,x)\equiv a(x)$ is independent of $t$,   the operator
  ${L}: C(\bar D)\to C(\bar D)$, $({L}u)(x)=\int_D \kappa(y-x)u(y)dy+a(x)u(x)$, may not have an eigenvalue associated with a positive eigenfunction  when   $a(x)$ is not a constant function (see \cite{Co1, ShZh1} for examples).
 Because of this, to study the aspects of the spectral theory for nonlocal dispersal operators,
   the concept of principal spectrum point for nonlocal dispersal operators
 was  introduced in \cite{HuShVi} (see also \cite{RaSh, ShXi1}), and the concept of generalized principal eigenvalues for nonlocal dispersal operators
 was  introduced in \cite{BeCoVo} (see also \cite{Co1}).  Some criteria have been  established in \cite{RaSh, ShZh1}
 for the principal spectrum point of a time periodic dispersal operator to be an eigenvalue with a positive eigenfunction.
 In \cite{Co1},  some criteria were  established for the generalized principal eigenvalue of a time independent dispersal operator to be an eigenvalue
 with a positive eigenfunction.

However, there is not much study on the aspects of spectral theory  for \eqref{main-linear-eq} when $a(t,x)$ is not periodic in $t$.
 In this first part of the series,
we investigate the spectral theory  for  \eqref{main-linear-eq} from two aspects:  top Lyapunov exponents and generalized principal eigenvalues. In particular, we provide various characterizations of the top Lyapunov exponents and generalized principal eigenvalues of \eqref{main-linear-eq}, discuss the relations between them, and study the effect of time and space variations of $a(t,x)$ on them.
  The theory of the top Lyapunov exponents and generalized principal eigenvalues is referred to as {\it  the principal spectral theory}
 for the nonlocal dispersal operators.  In the second part of the series, we will study the asymptotic dynamics of \eqref{main-nonlinear-eq} applying the principal spectral theory to be developed in this part.

In the rest of the introduction, we present  the notations and definitions in subsection 1.1, state the main results in subsection 1.2, and
make some remarks on the concepts and results in subsection 1.3.

\subsection{Notations and definitions}

Let
\begin{equation}
\label{X-space-eq}
X(D)= C_{\rm unif}^b(\bar D)=\{u\in C(\bar D)\,|\, u\,\,\, \text{is uniformly continuous and bounded}\}
\end{equation}
with norm $\|u\|=\sup_{x\in D}|u(x)|$.  If no confusion will occur, we may put
$$
X=X(D).
$$
 For any $s\in\RR$ and $u_0\in X$, let $u(t,x;s,u_0)$ be the unique  solution of
\eqref{main-linear-eq} with $u(s,x;s,u_0)=u_0(x)$ (the existence and uniqueness of solutions of \eqref{main-linear-eq}
with given initial function $u_0\in X$ follow from  the general semigroup theory, see \cite{Paz}). Put
\begin{equation}
\label{Phi-t-s-eq}
\Phi(t,s;a)u_0=u(t,\cdot;s,u_0).
\end{equation}

\begin{definition}
\label{top-lyapunov-exp-def} Let
\begin{equation}
\label{lyapunov-exp-eq}
\lambda_{PL}(a)=\limsup_{t-s\to\infty}\frac{\ln \|\Phi(t,s;a)\|}{t-s}, \quad \lambda_{PL}^{'}(a)=\liminf_{t-s\to\infty} \frac{\ln \|\Phi(t,s;a)\|}{t-s}.
\end{equation}
$\lambda_{PL}(a)$ and $\lambda_{PL}^{'}(a)$ are called the {\rm top Lyapunov exponents} of \eqref{main-linear-eq}.
\end{definition}

For given
$\lambda\in\R$, define
\begin{equation*}
\Phi_\lambda(t,s;a)=e^{-\lambda(t-s)}\Phi(t,s;a),
\end{equation*}
where $\Phi(t,s;a)$ is as in \eqref{Phi-t-s-eq}.

\begin{definition}
\label{exp-dichotomy-def}
Given $\lambda\in\RR$,
$\{\Phi_\lambda(t,s;a)\}_{s,t\in\RR,s\leq t}$
is said to admit an {\rm exponential dichotomy (ED} for short) on $X$ if
there exist $\beta>0$ and $C>0$ and continuous projections
$P(s):X\to X$ $(s\in\RR)$ such that for any $s,t\in\RR$ with $s\leq t$  the
following holds:
\begin{enumerate}
\item[(1)] $\Phi_\lambda(t,s;a)P(s)=P(t)\Phi_\lambda(t,s;a)$;

\item[(2)] $\Phi_\lambda(t,s;a)|_{R(P(s))}: R(P(s))\to R(P(t))$ is an
isomorphism for $t\geq s$ (hence $\Phi_\lambda(s,t;a):=\Phi_\lambda(t,s;a)^{-1}:$
$R(P(t))\to R(P(s))$ is well defined);
\item[(3)]
\begin{equation*}
\|\Phi_\lambda(t,s;a)(I-P(s))\|\leq C e^{-\beta (t-s)},\quad t\geq s
\end{equation*}
\begin{equation*}
\|\Phi_\lambda(t,s;a)P(s)\|\leq C e^{\beta (t-s)},\quad t\leq s.
\end{equation*}
\end{enumerate}
\end{definition}

\begin{definition}
\label{dynamic-spec-def}
\begin{enumerate}
\item[(1)] $\lambda\in\RR$ is said to be in the {\rm dynamical
spectrum}, denoted by $\Sigma(a)$,
of \eqref{main-linear-eq} or  $\{\Phi(t,s;a)\}_{s\le t}$
 if $\Phi_\lambda(t,s;a)$
does not admit an ED.

\item[(2)]  $\lambda_{PD}(a)=\sup\{\lambda\in\Sigma(a)\}$ is called
the {\rm principal dynamical spectrum point} of  $\{\Phi(t,s;a)\}_{s\le t}$.
\end{enumerate}
\end{definition}

Throughout this paper, we say  that a {\it property holds for a function $u(t,x)$ for a.e. $t\in I\subset \RR$ and all $x\in E\subset \RR^N$}
if there is a subset $I_0$ of $I$ with zero Lebesgue measure such that the property holds for $u(t,x)$ for all
$(t,x)\in (I\setminus I_0)\times E$.

Let
\begin{equation}
\label{X-script-space-eq}
\mathcal{X}(D)=C_{\rm unif}^b(\RR\times\bar D):=\{u\in C(\RR\times \bar D)\,|\, u\,\, \text{is uniformly continuous and bounded}\}
\end{equation}
with the norm $\|u\|=\sup_{(t,x)\in\RR\times \bar D}|u(t,x)|$. In the absence of possible confusion, we may write
$$
\mathcal{X}=\mathcal{X}(D).
$$
Let $L(a):\mathcal{D}(L(a))\subset \mathcal{X}\to \mathcal{X}$
be defined as follows,
$$
(L(a)u)(t,x)=-\p_t u(t,x)+\int_D
\kappa(y-x)u(t,y)dy+a(t,x)u(t,x).
$$
Let
\begin{align*}
\Lambda_{PE}(a)=\{\lambda\in\RR\,|\, \exists \, \phi\in \mathcal{X},\,\, \inf_{t\in\RR} \phi(t,x)\ge \not\equiv 0,\,\,  (L(a)\phi)(t,x)\ge\lambda \phi(t,x)\,\, {\rm for}\,\, a.e.\, t\in \RR,\,\, {\rm all}\,\,   x\in \bar D\}
\end{align*}
and
\begin{align*}
\Lambda_{PE}^{'}(a)=\{\lambda\in\RR\,|\,
\exists \, \phi\in \mathcal{X},\,\,  \inf_{t\in\RR,x\in\bar D} \phi(t,x)>0,\,\,(L(a)\phi)(t,x)\le\lambda \phi(t,x)\,\, {\rm for}\,\, a.e.\, t\in \RR,\,\, {\rm all}\, \,  x\in \bar D\}.
\end{align*}

\begin{definition}
\label{generalized-principal-eigenvalue-def} Define
\begin{equation}
\label{generalized-principal-eigen-eq}
\lambda_{PE}(a)=\sup\{\lambda\,|\, \lambda\in\Lambda_{PE}(a)\}
\end{equation}
and
\begin{equation}
\label{generalized-principal-eigen-eq1}
\lambda_{PE}^{'}(a)=\inf\{\lambda\,|\, \lambda\in\Lambda_{PE}^{'}(a)\}.
\end{equation}
Both $\lambda_{PE}(a)$ and $\lambda_{PE}^{'}(a)$ are called {\rm generalized principal eigenvalues} of \eqref{main-linear-eq}.
\end{definition}

Let
\begin{equation}
\label{time-average-eq}
\hat a(x)=\lim_{T\to\infty}\frac{1}{T}\int_0^T a(t,x)dt
\end{equation}
((see Proposition \ref{almost-periodic-prop} for the existence of
$\hat a(\cdot)$). Let
\begin{equation}
\label{time-space-average1}
\bar a= \frac{1}{|D|}\int_D \hat a(x)dx
\end{equation}
when $D$ is bounded, and
\begin{equation}
\label{time-space-average2}
\bar a=\lim_{q_1,q_2,\cdots,q_N\to\infty}\frac{1}{q_1q_2\cdots q_N}\int_{0}^{q_N}\cdots\int_0^{q_2}\int_0^{q_1} \hat a(x_1,x_2,\cdots,x_N)dx_1dx_2\cdots dx_N
\end{equation}
when $D=\RR^N$ and $a(t,x)$ is almost periodic in $x$ uniformly with respect to $t\in\RR$ (see Proposition \ref{almost-periodic-prop} for the existence of
 $\bar a$). Note that $\hat a(x)$ is the time average of $a(t,x)$, and $\bar a$ is the space average
of $\hat a(x)$.

To discuss the monotonicity of $\lambda_{PL}(a)$, $\lambda_{PE}(a)$, and $\lambda_{PE}^{'}(a)$ with respect to the domain $D$, we may put
$$\Phi(t,s;a,D)=\Phi(t,s;a),\quad \Lambda_{PE}(a,D)= \Lambda_{PE}(a), \quad \Lambda_{PE}^{'}(a,D) = \Lambda_{PE}^{'}(a).
$$
and
$$
\lambda_{PL}(a,D)=\lambda_{PL}(a),\,\, \lambda_{PE}(a,D)=\lambda_{PE}(a),\,\,  \lambda_{PE}^{'}(a,D) =\lambda_{PE}^{'}(a).
$$

\subsection{Main results}

In this subsection, we state the main theorems of this paper. Throughout this subsection, we assume that
 $a(t,x)$ satisfies {\bf (H2)}. Sometime, we may also assume that

\medskip

\noindent {\bf (H2$)^{'}$}  $a(t,x)$ is limiting almost periodic in $t$ with respect to $x$ and is also limiting almost periodic in $x$ when $D=\RR^n$ (see Definition \ref {almost-periodic-def}).

The first theorem is on the relation between $\lambda_{PL}^{'}(a)$,
$\lambda_{PL}(a)$, and $\lambda_{PD}(a)$.

\begin{theorem}[Relations between $\lambda_{PL}^{'}(a)$, $\lambda_{PL}(a)$ and $\lambda_{PD}(a)$]
\label{relation-thm1}
$\quad$

\begin{itemize}
\item[(1)] For any $u_0\in X$ with
$\inf_{x\in D}u_0(x)>0$,
$$
\lambda_{PL}^{'}(a)=\lambda_{PL}(a)=\lim_{t-s\to\infty} \frac{\ln\|\Phi(t,s;a)\|}{t-s}=\lim_{t-s\to\infty} \frac{\ln\|\Phi(t,s;a)u_0\|}{t-s}.$$

\item[(2)]
$\lambda_{PL}(a)=\lambda_{PD}(a).$
\end{itemize}
\end{theorem}

The second theorem is on the relations between $\lambda_{PE}(a)$, $\lambda_{PE}^{'}(a)$, and $\lambda_{PL}(a)$.

\begin{theorem}[Relations  between $\lambda_{PE}(a)$, $\lambda_{PE}^{'}(a)$, and $\lambda_{PL}(a)$]
\label{relation-thm2}

$\quad$

\begin{itemize}
\item[(1)]
$
\lambda_{PE}^{'}(a)=\lambda_{PL}(a).
$

\item[(2)]
$
\lambda_{PE}(a)\le \lambda_{PL}(a).
$
If $a(t,x)$  satisfies (H2$)^{'}$, then
$\lambda_{PE}(a)=\lambda_{PL}(a).
$

\item[(3)] If $a(t,x)\equiv a(t)$, then $\lambda_{PE}(a)=\lambda_{PE}^{'}(a)=\lambda_{PL}(a)=\hat a+\lambda_{PL}(0)$.
\end{itemize}
\end{theorem}

The third theorem is on the effects of time and space variations on $\lambda_{PE}(a)$.

\begin{theorem}[Effects of time and space variations on $\lambda_{PE}(a)$]
\label{variation-thm1}

$\quad$

\begin{itemize}
\item[(1)] $\lambda_{PE}(a)\ge \sup_{x\in D}\hat a(x)$. If $a(t,x)$ satisfies (H2$)^{'}$, then
$\lambda_{PE}(a) \ge \lambda_{PE}(\hat{a}) \ge \sup_{x\in D}\hat a(x).$

\item[(2)] If $D$ is bounded,  $a(t,x)\equiv a(x)$, and $\kappa(\cdot)$ is symmetric, then
$$\lambda_{PE}(a)\ge \bar a+\frac{1}{|D|}\int_D\int_D \kappa(y-x)dydx,
$$
 where $|D|$ is the Lebesgue measure of $D$.

 \item[(3)]  If $D=\RR^N$, $a(t,x)\equiv a(x)$  is almost periodic in $x$, and
 $\kappa(\cdot)$ is symmetric,  then
 $$\lambda_{PE}(a)\ge \bar a+1.
 $$
\end{itemize}
\end{theorem}

The fourth theorem is on the effects of time and space variations on $\lambda_{PL}(a)$.

\begin{theorem}[Effects of time and space variations on $\lambda_{PL}(a)$]
\label{variation-thm2}

$\quad$

\begin{itemize}
\item[(1)]  If $D$ is bounded or $D=\RR^N$ and $a$ satisfies (H2$)^{'}$, then  $\lambda_{PL}(a)\ge { \lambda_{PL}(\hat a)}\ge \sup_{x\in D}\hat a(x)$.

\item[(2)] If $D$ is bounded and $\kappa(\cdot)$ is symmetric,  then
$$\lambda_{PL}(a)\ge \bar a+\frac{1}{|D|}\int_D\int_D \kappa(y-x)dydx.
$$

\item[(3)] If $D=\RR^N$,  $a(t,x)$ is almost periodic in $x$ uniformly with respect to $t\in\RR$,  and $\kappa(\cdot)$ is symmetric,  then
$$\lambda_{PL}(a)\ge \bar a+1.$$
\end{itemize}
\end{theorem}

The last theorem is on the characterization of $\lambda_{PE}(a)$ and $\lambda_{PE}^{'}(a)$ when $a(t,x)$ is independent of $t$ or periodic in $t$.

\begin{theorem}[Characterization of $\lambda_{PE}(a)$ and $\lambda_{PE}^{'}(a)$]  Assume that $a$ satisfies (H2$)^{'}$.
\label{characterization-thm1}
\begin{itemize}
\item[(1)] If $a(t,x)\equiv a(x)$, then
$$
 \lambda_{PE}(a)=\sup\{\lambda\,|\, \lambda\in \tilde \Lambda_{PE}(a)\}=\inf\{\lambda\,|\,
 \lambda\in\tilde \Lambda_{PE}^{'}(a)\}=\lambda_{PE}^{'}(a),
 $$
 where
\begin{align*}
\tilde\Lambda_{PE}(a)=\{\lambda\in\RR\,|\, \exists \, \phi\in {X},\,\,  \phi(x)\ge \not\equiv  0,\,\, \int_D \kappa(y-x)\phi(y)dy+a(x)\phi(x)\ge\lambda \phi(x)\,\, \forall\, x\in \bar D\}
\end{align*}
and
\begin{align*}
\tilde \Lambda_{PE}^{'}(a)= \{\lambda\in\RR\,|\, \exists \, \phi\in {X},\,\, \inf_{x\in D} \phi(x)>0,\,\, \int_D \kappa(y-x)\phi(y)dy+a(x)\phi(x)\le \lambda \phi(x)\,\, \forall\, x\in \bar D\}.
\end{align*}

\item[(2)] If $a(t+T,x)\equiv a(t,x)$, then
$$
 \lambda_{PE}(a)=\sup\{\lambda\,|\, \lambda\in \hat \Lambda_{PE}(a)\}=\inf\{\lambda\,|\,
 \lambda\in\hat \Lambda_{PE}^{'}(a)\}=\lambda_{PE}^{'}(a),
 $$
 where
\begin{align*}
\hat \Lambda_{PE}(a)=\{\lambda\in\RR\,|\, \exists \, \phi\in \mathcal{X}_T,\,\, \inf_{t\in\RR} \phi(t,x)\ge \not\equiv  0,\,\,  (L(a)\phi)(t,x)\ge\lambda \phi(t,x)\,\, {\rm for}\,\, a.e.\, t\in \RR,\,\, {\rm all}\, \,  x\in \bar D\},
\end{align*}
\begin{align*}
\hat \Lambda_{PE}^{'}(a)=\{\lambda\in\RR\,|\,
\exists \, \phi\in \mathcal{X}_T,\,\,  \inf_{t\in\RR,x\in D} \phi(t,x)>0,\,\,(L(a)\phi)(t,x)\le\lambda \phi(t,x)\,\, {\rm for}\,\, a.e.\, t\in \RR,\,\, {\rm all}\,\,  x\in \bar D\}.
\end{align*}
and
$$
\mathcal{X}_T=\{\phi\in \mathcal{X}\,|\, \phi(t+T,x)=\phi(t,x)\}.
$$
\end{itemize}
\end{theorem}

\subsection{Remarks}

In this subsection, we provide the following remarks on the main results established in this paper.

\begin{itemize}

\item[1.] Spectral theory for a linear evolution equation is strongly related to the growth/decay rates of its solutions.
    From the point of view of dynamical systems, one usually employ  the top Lyapunov exponents and principal dynamical spectrum point to characterize  the largest  growth rate of the solutions of a linear evolution equation.
    Theorem \ref{relation-thm1} shows that the top Lyapunov exponents and principal dynamical spectrum point of \eqref{main-linear-eq} are the same, which is then exactly the largest growth rate of the solutions of \eqref{main-linear-eq}.

\item[2.] The notion of generalized principal eigenvalues for time independent nonlocal dispersal equations  was introduced in
\cite{BeCoVo, Co1, CoDaMa1} (see item 3 in the following for some detail). It is a natural extension of principal eigenvalues, which is related
  to the existence of  eigenvalues associated with  positive eigenfunctions.
   Theorem \ref{relation-thm2} shows that
  $$
  \lambda_{PE}(a)=\lambda_{PE}^{'}(a)=\lambda_{PL}(a)
  $$
  when $a(t,x)$ is limiting almost periodic in $t$, and in general,
 $$
  \lambda_{PE}(a)\le \lambda_{PE}^{'}(a)=\lambda_{PL}(a).
  $$
  Therefore, in any case, $\lambda_{PE}^{'}(a)$ is exactly the largest growth rate of the solutions of \eqref{main-linear-eq}.
  It is definitely of great importance that the largest growth rate of the solutions of \eqref{main-linear-eq} can be characterized by two different approaches, one by the top Lyapunov exponent $\lambda_{PL}(a)$ and the other by the generalized
 principal eigenvalue $\lambda_{PE}^{'}(a)$.

  \item[3.] When $a(t,x)\equiv a(x)$, the following generalized principal  eigenvalues were introduced  in \cite{BeCoVo}
  for \eqref{main-linear-eq}:
  $$
  \lambda_p(a)=\sup\{\lambda\in\RR\,|\, \exists \phi\in C(\bar D),\, \phi>0,\,\, \int_D \kappa(y-x)\phi(y)dy+a(x)\phi+\lambda\phi\le 0\,\, {\rm in}\,\, D\},
  $$
  and
  $$
  \lambda_p^{'}(a)=\inf\{\lambda\in\RR\,|\, \exists \, \phi\in C(D)\cap L^\infty(D),\,\, \phi\ge \not\equiv 0,\,\,
  \int_D\kappa(y-x)\phi(y)dy+a(x)\phi+\lambda\phi\ge 0\,\, {\rm in}\,\, D\}.
  $$
  Note that, in our definitions of $\lambda_{PL}(a)$ and $\lambda_{PE}^{'}(a)$, we require the function $\phi$
  in the sets $\Lambda_{PL}(a)$ and $\Lambda_{PE}^{'}(a)$ to be uniformly continuous and bounded.
  By Theorem \ref{characterization-thm1}, we have the following relation between $\lambda_p(a), \lambda_p^{'}(a)$, and $\lambda_{PE}(a), \lambda_{PE}^{'}(a)$:
  \begin{equation}
  \label{pl-pe-eq0}
 -\lambda_p(a)\le \lambda_{PE}^{'}(a)=  \lambda_{PE}(a)\le -\lambda_p^{'}(a),
  \end{equation}
  which implies that
  \begin{equation}
  \label{pl-pe-eq00}
  \lambda_p^{'}(a)\le \lambda_p(a).
  \end{equation}
 It should be pointed out that, among others, it was  proved in \cite{BeCoVo} that, if $\kappa(\cdot)$ has compact support, then
  \begin{equation}
  \label{pl-pe-eq000}
  \lambda_p(a)=\lambda_p^{'}(a)\quad {\rm when}\,\, D \,\, {\rm is\,\,\, bounded}
  \end{equation}
  and
  \begin{equation}
  \label{pl-pe-eq0000}
  \lambda_p^{'}(a)\le \lambda_p(a)\quad {\rm when}\,\, D\,\, {\rm is\,\,\, unbounded}
  \end{equation}
  (see \cite[Theorem 1.1]{BeCoVo}  and \cite[Theorem 1.2]{BeCoVo}).
   It should also be pointed out that the paper \cite{BeCoVo} dealt  with more general kernel functions $\kappa(x,y)$.
    Note that  in Theorem \ref{relation-thm2}(2), it was proved that  \eqref{pl-pe-eq00} holds without the assumption that $\kappa(\cdot)$ has compact support. Hence \eqref{pl-pe-eq00} is an improvement of \eqref{pl-pe-eq0000} when the kernel function in \cite{BeCoVo}  $\kappa(x,y)=\kappa(y-x)$.

\item[4.] Theorems \ref{variation-thm1} and \ref{variation-thm2} are on the influence of  time and space variation of
$a(t,x)$ on the top Lyapunov exponent $\lambda_{PL}(a)$  and the generalized principal eigenvalue $\lambda_{PE}(a)$.
Theorem \ref{variation-thm2}(1) shows that  time variation does not reduce the top Lyapunov exponent $\lambda_{PL}(a)$.
Since $\lambda_{PE}^{'}(a)=\lambda_{PL}(a)$, this also holds for $\lambda_{PE}^{'}(a).$
Theorem \ref{variation-thm1}(2)   indicates that  space variation of
$a(t,x)\equiv a(x)$ does not reduce the generalized principal eigenvalue $\lambda_{PE}(a)$  when \eqref{main-linear-eq} is viewed as a nonlocal dispersal equation with Neumann type boundary condition on the bounded domain $D$. To be more precise, write \eqref{main-linear-eq} with $a(t,x)\equiv a(x)$ as
\begin{equation}
\label{main-linear-eq1}
u_t=\int_D \kappa(y-x)[u(t,y)-u(t,x)]dy +\tilde a(x)u(t,x),\quad x\in \bar D,
\end{equation}
where $\tilde a(x)=\int_D \kappa(y-x)dy+a(x)$. \eqref{main-linear-eq1} can then be viewed as a nonlocal dispersal equation with reaction term $\tilde a(x) u$ and Neumann type boundary condition.  Theorem \ref{variation-thm1}(2) then follows from
the arguments of
\cite[Theorem 2.1]{ShXi1}.  Note that the random dispersal counterpart of \eqref{main-linear-eq1} is the following reaction diffusion equation on $D$ with Neumann boundary condition,
$$
\begin{cases}
u_t=\Delta u+\tilde a(x)u,\quad x\in D\cr
\frac{\p u}{\p n}=0,\quad x\in\p D.
\end{cases}
$$
Theorem \ref{variation-thm1}(3)  indicates that the space variation of
$a(t,x)\equiv a(x)$ does not reduce the generalized principal eigenvalue $\lambda_{PE}(a)$  when \eqref{main-linear-eq}
 with $a(t,x)\equiv a(x)$ is viewed as the following nonlocal dispersal equation on $\RR^N$ with reaction term $\tilde a(x)u$,
\begin{equation}
\label{main-linear-eq2}
u_t=\int_{\RR^N} \kappa(y-x)(u(t,y)-u(t,x))dy+\tilde a(x)u,\quad x\in\RR^N,
\end{equation}
where $\tilde a(x)=1+a(x)$. When $a(x)$ is periodic in $x$, Theorem \ref{variation-thm1}(3) follows from the arguments of \cite[Theorem 2.1]{HeShZh}. When $a(x)$ is almost periodic in $x$, Theorem \ref{variation-thm1}(3) is new.
Note that Theorem \ref{variation-thm2}(2), (3) follow from Theorem \ref{variation-thm1}(2), (3) and the fact that
$\lambda_{PL}(a)\ge \lambda_{PE}(a)$.

\item[5.] There are several interesting open problems. For example,
 it remains open whether $\lambda_{PE}(a)=\lambda_{PE}^{'}(a)$   for any $a$ satisfying {\bf (H2)}. If $\lambda_{PE}(a)=\lambda_{PE}^{'}(a)$,   under what condition
  there is a positive function $\phi(t,x)$, such that
 $$
 -\phi_t+\int_D\kappa(y-x)\phi(t,y)dy+a(t,x)\phi(t,x)=\lambda_{PE}(a)\phi(t,x)\quad \forall\,\, t\in\RR,\,\, x\in \bar D.
 $$
 If there is such $\phi(t,x)$, we may call $\lambda_{PE}(a)$ the {\it principal eigenvalue} of \eqref{main-linear-eq}. It remains open whether $\lambda_{PE}(a)\ge \lambda_{PE}(\hat{a})$ for any $a$ satisfying {\bf (H2)}.

\item[6.] It should be pointed out that the definitions of top Lyapunov exponents, principal dynamical spectrum point, and
generalized principal eigenvalues can be applied to \eqref{main-linear-eq} when $a(t,x)$ is a general time dependent function.
But some results in the above theorems may not hold when $a(t,x)$ is not almost periodic in $t$, for example,
$\lambda_{PL}^{'}(a)=\lambda_{PL}(a)$ may not be true when $a(t,x)$ is not almost periodic in $t$. The aspects of spectral theory of \eqref{main-linear-eq} with general time dependent $a(t,x)$ will not be discussed in this paper.
\end{itemize}

The rest of the paper is organized as follows: In section 2 we present some preliminary materials  to be used in the  proofs of the main results. In section 3 we study the top Lyapunov exponents
 of \eqref{main-linear-eq} and prove Theorem \ref{relation-thm1}. In section 4, we explore the relations between the top Lyapunov exponents and generalized principal eigenvalues, and prove Theorem \ref{relation-thm2}.   We discuss the effects of time and space variations on the generalized principal eigenvalue  $\lambda_{PE}(a)$ and
prove Theorem \ref{relation-thm2} in section 5. We consider the effects of space and time variations on  the top Lyapunov exponent  $\lambda_{PL}(a)$ and
 prove Theorem \ref{variation-thm2} in section 6. In the last section, we provide some characterization for the generalized principal eigenvalues  $\lambda_{PE}(a)$ and $\lambda_{PE}^{'}(a)$ when $a(t,x)$ is independent of $t$ or periodic in $t$ and prove Theorem \ref{characterization-thm1}.

\section{Preliminary}

 In this section, we collect  some preliminary materials to be used in the proofs of Theorems \ref{relation-thm1}-\ref{characterization-thm1} in later sections.

First, we present the definitions of almost periodic functions and limiting almost periodic functions, and some basic properties
of almost periodic functions.

\begin{definition}
\label{almost-periodic-def}
\begin{itemize}
\item[(1)]
Let  $E\subset\RR^N$ and $f \in C(\mathbb{R} \times E, \mathbb{R})$. $f(t,x)$ is   said to be
 {\rm almost periodic in t uniformly with respect to $x \in E$} if it is uniformly continuous in $(t,x)\in \RR\times E$ and
  for any $\epsilon > 0$, $T(\epsilon)$ is relatively dense in $\RR$, where
$$
T(\epsilon)=\{\tau\in\RR\,|\,   |f(t + \tau, x) - f(t,x) | \le \epsilon\,\, \forall\, t\in\RR,\, x\in E\}.
 $$

\item[(2)] Let $E\subset\RR^N$ and $f\in C(\RR\times E,\RR)$. $f$ is said to be {\rm limiting almost periodic in $t$ uniformly with respect to $x\in E$}  if there is a sequence $f_n(t,x)$ of uniformly continuous functions which are periodic in $t$ such that $$
    \lim_{n\to\infty} f_n(t,x)=f(t,x)
    $$
    uniformly in $(t,x)\in\RR\times E$.

\item[(3)] Let $f \in C(\RR\times \RR^N,\RR)$.  $f(t,x)$ is said to be {\rm almost periodic in $x$ uniformly with respect to $t\in\RR$} if $f$ is uniformly continuous in $(t,x)\in\RR\times\RR^N$ and for each $1\le i\le N$, $f(t,x_1,x_2,\cdots,x_N)$ is almost periodic
    in $x_i$ uniformly with respect to $t\in\RR$ and $x_j\in\RR$ for $1\le j\le N, j\not =i$.
    \end{itemize}
\end{definition}

\begin{proposition}
\label{almost-periodic-prop}
\begin{itemize}

\item[(1)] If $f(t,x)$ is almost periodic in $t$ uniformly with respect to $x\in E$, then for any sequence $\{t_n\}\subset\RR$, there is a subsequence $\{t_{n_k}\}$ such that the limit $\lim_{k\to\infty} f(t+t_{n_k},x)$ exists uniformly in
    $(t,x)\in\RR\times E$.

\item[(2)] If $f(t,x)$ is almost periodic in $t$ uniformly with respect to $x\in E$, then
 the limit
 $$\hat f(x):=\lim_{T\to\infty}\frac{1}{T}\int_0^T f(t,x)dt
 $$
   exists uniformly with respect to $x\in E$.
 If $E=\RR^N$  and for each $1\le i\le N$, $f(t,x_1,x_2,\cdots,x_N)$ is also almost periodic in $x_i$ uniformly with respect to
 $t\in\RR$  and $x_j\in\RR$ for $1\le j\le N$, $j\not =i$, then  the limit
 $$
 \bar f:=\lim_{q_1,q_2,\cdots,q_N\to \infty}\frac{1}{q_1 q_2\cdots q_N}\int_0^{q_N}\cdots\int_0^{q_2}\int_0^{p_N} \hat f(x_1,x_2,\cdots,x_N)dx_1dx_2\cdots  dx_N$$
  exists.

    \item[(3)]  Given an almost periodic function $f(t)$, for any $\epsilon > 0,$ there exists a  trigonometric polynomial $P_{\epsilon}(t) = \sum_{k=1}^{N_{\epsilon}}b_{k,\epsilon}e^{i\lambda_{k,\epsilon}t}$ such that $$\underset{t\in\RR}{\sup}\;\|f(t) - P_{\epsilon}(t)\| < \epsilon.$$
\end{itemize}
\end{proposition}

\begin{proof}
(1) It follows from \cite[Theorem 2.7]{fink}

(2) It follows from \cite[Theorem 3.1]{fink}

(3) It  follows from \cite[Theorem 3.17]{fink}.
\end{proof}

Next, we introduce the concept of sub- and super-solutions of
\eqref{main-linear-eq} and   present some comparison principle for (\ref{main-linear-eq}).

Recall that
$$
X=X(D):= C_{\rm unif}^b(\bar D)=\{u\in C(\bar D)\,|\, u\,\,\, \text{is uniformly continuous and bounded}\}.
$$
For given $u^1,\; u^2 \in X,$ we define
$$u^1 \le u^2, ~~\text{if} ~~  u^1(x)\le u^2(x)\quad \forall\, x\in\bar D.$$

\begin{definition}
A continuous function $u(t,x)$ on $[0,\tau)\times\bar D$ is called a super-solution (or sub-solution) of  \eqref{main-linear-eq} if  $u(t,x)$ is differentiable in $t$ for a.e.  $t\in [0,\tau)$ and all $x\in\bar D$  and satisfies,
$$\frac{\partial u}{\partial t} \geq (or \leq)
\int_D \kappa(y-x)u(t,y)dy+a(t,x)u(t,x)\quad \text{for} \;\; a.e. \, t\in [0,\tau),\,\, {\rm all}\, \, x\in\bar D.
$$
\end{definition}

\begin{proposition}
\label{comparison-prop}
(Comparison Principle)
\begin{itemize}
    \item[(1)]  If $u^1(t,x)$ and $u^2(t,x)$ are bounded sub- and super-solutions of \eqref{main-linear-eq} on $[0, \tau)$ and $u^1(0,\cdot) \leq u^2(0,\cdot)$, then $u^1(t,\cdot) \leq u^2(t,\cdot)$ for $t \in [0,\tau)$.

    \item[(2)] For given $u_0 \in X$ with $0\le u_0$,  and $a^1,\:a^2 \in \mathcal{X}$,  if $a^1 \leq a^2$ then $u(t,\cdot ; s,u_0, a^1) \leq u(t,\cdot ; s,u_0, a^2)$, where $u(t,\cdot;s,u_0,a_i)$ is the solution of \eqref{main-linear-eq} with
        $a$ being replaced by $a_i$ and $u(s,\cdot;s,u_0,a_i)=u_0(\cdot)$ for $i=1,2$.
 \end{itemize}
\end{proposition}

\begin{proof}
(1) This follows from the arguments of Proposition 2.1 of \cite{ShZh1}. For the sake of completeness, we provide a proof.

 Set $v(t,x) = e^{ct}(u^2(t,x) - u^1(t,x))$. Then $v(t,x)$ satisfies
\begin{equation}\label{veqn}
\frac{\partial v}{\partial t} \ge \int_D \kappa(y-x)v(t,y)dy+p(t,x)v(t,x) ~~\text{for}~~
a.e. \;\;t\in [0,\tau) \;\; and ~~all ~~x\in \bar{D},
\end{equation}
where $p(t,x)=a(t,x) + c$.  Take $c>0$ such that $p(t,x) >0$ for all $t \in \RR$ and  $x \in D$.
Let $p_0 = \underset{t\in \RR,  x\in D}{\sup}p(t,x)$ and $T_0 = \min\{\tau, \frac{1}{p_0 + 1}\}.$

  Assume that there exist  $\Bar{t}\in (0,T_0) $ and $\bar{x} \in D$ such that $v(\Bar{t}, \Bar{x}) < 0.$ Then there exists $t_0 \in (0, T_0)$ such that $v_{inf} := \underset{(t,x) \in [0,t_0) \times D}{\inf}v(t,x)<0$. We can  then find $t_n \in [0,t_0), \; x_n \in D$ such that  $v(t_n, x_n) \to v_{inf}$ as $n \to \infty.$ By (\ref{veqn}), we have
$$
v(t_n, x_n) - v(0, x_n)  \ge  \int_0^{t_{n}}[\int_D \kappa(y-x_n)v(t,y)dy+p(t,x_n)v(t,x_n)]dt.
$$
By $v(0, x_n) \ge 0$,  we have
$$ v(t_n, x_n) \ge  \int_0^{t_n}[\int_D \kappa(y-x)v_{inf}dy+ p_0v_{inf}]dt + v(0, x_n)
\ge t_n(1+p_0)v_{inf}.$$
This implies that
$$v_{inf} \ge t_0(1+p_0)v_{inf} > v_{inf},
 $$
 which is a contradiction. Hence $v(t,x) \ge 0$ for all $t \in [0, T_0)$ and for all $x \in D.$

 Let $k\ge 1$ be such that $kT_0\le \tau$ and $(k+1)T_0>\tau$.
 Repeat the above arguments, we have
 $$
 v(t,x)\ge 0\quad \forall\, t\in [(i-1)T_0, iT_0),\,\, x\in \bar D,\,\, i=1,2, \cdots k,
 $$
 and
 $$
 v(t,x)\ge 0\quad \forall \, t\in [kT_0, \tau),\,\, x\in \bar D.
 $$
 It then follows that
 $$
 v(t,x)\ge 0\quad \forall\, t\in[0,\tau),\,\, x\in\bar D.
 $$
This implies that $u^1(t,x) \le u^2(t,x)$ for all $t \in [0,\tau), \; x \in D.$

(2) By (1),
$$
u(t,x;s,u_0,a^i)\ge 0\quad \forall\, t\ge 0,\,\, x\in\bar D,\,\,i=1,2.
$$
This together with $a^1\le a^2$ implies that
$$
u_t(t,x;s,u_0,a^1)\le \int_D \kappa(y-x)u(t,y;s,u_0,a^1)dy+a^2(t,x)u(t,x;s,u_0,a^1)\quad \forall\, t\ge 0,\,\, x\in\bar D.
$$
Then by (1) again,
$$
u(t,x;s,u_0,a^1)\le u(t,x;s,u_0,a^2)\quad \forall\, t\ge 0,\,\, x\in\bar D.
$$
The proposition is thus proved.
\end{proof}

Finally, we  recall some existing results on the principal eigenvalue theory for \eqref{main-linear-eq} when
$D$  is bounded and $a(t,x)$ is $T$-periodic in $t$ (i.e. $a(t+T,x)=a(t,x)$)
 or $D=\RR^N$ and $a(t,x)$ is $T$-periodic in $t$ and $P$-periodic in $x$, where $P=(p_1,p_2,\cdots,p_N)$ and $p_i\ge 0$ for $i=1,2,\cdots,N$) (i.e. $a(t+T,x)=a(t,x+p_i{\bf e_i})=a(t,x)$
 for $i=1,2,\cdots,N$).

Let
$$
X_P=\begin{cases}
X\quad {\rm if}\quad D \quad {\rm is \,\, bounded}\cr
\{u\in X\,|\, u\,\, \text{is}\,\, P{\rm -periodic\,\,  in} \,\, x\}\quad {\rm if}\quad
D=\RR^N.
\end{cases}
$$
Let
$$
\mathcal{X}_P=\begin{cases}
\{u\in \mathcal{X}\,|\, u \,\, {\rm is}\,\, T{\rm -periodic\,\, in}\,\, t\}\quad {\rm if}\quad D \quad {\rm is \,\, bounded}\cr
\{u\in \mathcal{X}\,|\, u\,\, \text{is}\,\, {\rm is}\,\, T{\rm -periodic\,\, in}\,\, t\,\, {\rm and}\,\, P{\rm -periodic\,\,  in} \,\, x\}\quad {\rm if}\quad
D=\RR^N.
\end{cases}
$$
For given $a\in\mathcal{X}_p$,
define  ${L}_p(a):\mathcal{D}({L}_p(a))\subset \mathcal{X}_P\to \mathcal{X}_P$  by
$$
{L}_p(a) u=-u_t+\int_{D}\kappa(y-x)u(t,y)dy+a(t,x)u.
$$

\begin{definition}
\label{principal-eigenvalue-def}
For given $a\in\mathcal{X}_p$, let
$$
\lambda_s(a)=\sup\{{\rm Re}\lambda\,|\, \lambda\in\sigma({L}_p(a))\},
$$
where  $\sigma({L}_p(a))$ is the spectrum of ${L}_p(a)$.
$\lambda_s(a)$ is called the {\rm principal spectrum
point} of $L_p(a)$. If $\lambda_s(a)$ is an isolated
eigenvalue of $L_p(a)$ with a positive eigenfunction $\phi$ (i.e.
$\phi\in \mathcal{X}_p$ with $\phi(t,x)>0$), then $\lambda_s(a)$ is called the {\rm
principal eigenvalue} of $L_p(a)$ or it is said that {\rm
$L_p(a)$ has a principal eigenvalue}.
\end{definition}

\begin{proposition}
\label{periodic-prop}
For given $a\in\mathcal{X}_p$, the following hold.
\begin{itemize}
\item[(1)] $\lambda_s(a) = \lambda_{PL}(a)$.

\item[(2)]The  principal eigenvalue of ${L}_p(a)$ exists
if $\hat a(\cdot)$ is $C^N$, there is some $x_0\in {\rm Int}(D)$
 satisfying $\hat a(x_0)=\max_{x\in \bar D}\hat a(x)$,
and  the partial derivatives of $\hat a(x)$ up to order
$N-1$ at $x_0$ are zero.

\item[(3)] For any
$\epsilon>0$, there is $a_{\epsilon}\in\mathcal{X}_P$ satisfying that
$$
\|a-a_{\epsilon}\|_{\mathcal{X}}<\epsilon;
$$
$\hat a_{\epsilon}$ is $C^N$;  $\hat a_{\epsilon}$
attains its maximum at some point $x_0\in {\rm Int}(D)$;  and the
partial derivatives of $\hat a_{\epsilon}$ up to order $N-1$
at $x_0$ are zero, where $\hat a_{\epsilon}(x)=\frac{1}{T}\int_0^T
a_{\epsilon}(t,x)dt$.

\item[(4)] $\lambda_s(a) \ge \lambda_s(\hat{a})  \ge \underset{x\in D}{\sup}\;\hat{a}(x).$
\end{itemize}
\end{proposition}

\begin{proof}
(1) It follows from \cite[Theorem 3.2]{HuShVi}.

(2) It follows from \cite[Theorem B(1)]{RaSh}.

 (3) It follows from   \cite[Lemma 4.1]{RaSh}.

 (4) It follows from  \cite[Theorem C]{RaSh}.
  \end{proof}


\section{Proof of Theorem \ref{relation-thm1}}

In this section, we
prove Theorem \ref{relation-thm1}. We first prove a lemma on the continuity of $\lambda_{PL}(a)$, $\lambda_{PL}^{'}(a)$,
$\lambda_{PE}(a)$, and
$\lambda_{PE}^{'}(a)$ in $a$.

\begin{lemma}
\label{continuity-lem1}
$\lambda_{PL}(a)$, $\lambda_{PL}^{'}(a)$,  $\lambda_{PE}(a)$, and $\lambda_{PE}^{'}(a)$ are continuous in $a\in\mathcal{X}$ satisfying (H2).
\end{lemma}

\begin{proof}
First,
 we prove the continuity of $\lambda_{PL}(a)$ and $\lambda_{PL}^{'}(a)$ in $a$. For any $a_1,a_2\in\mathcal{X}$ satisfying (H2),
$$
a_2(t,x)-\|a_2-a_1\|\le a_1(t,x)\le a_2(t,x)+\|a_2-a_1\|\quad \forall\, t\in\RR,\, x\in\bar D.
$$
This implies that  for any $u_0\in X$ with $u_0\ge 0$,  using Proposition \ref{comparison-prop} we have,
$$
e^{-\|a_2-a_1\| (t-s)}\Phi(t,s;a_2)u_0\le \Phi(t,s;a_1)u_0\le e^{\|a_2-a_1\|(t-s)}\Phi(t,s;a_2)u_0.
$$
It then follows that
$$
-\|a_2-a_1\|+\lambda_{PL}(a_2)\le\lambda_{PL}(a_1)\le \|a_2-a_1\|+\lambda_{PL}(a_2),
$$
and
$$
-\|a_2-a_1\|+\lambda_{PL}^{'}(a_2)\le\lambda_{PL}^{'}(a_1)\le \|a_2-a_1\|+\lambda_{PL}^{'}(a_2).
$$
Hence $\lambda_{PL}(a)$ and $\lambda_{PL}^{'}(a)$ are continuous in $a$.

Next, we prove that $\lambda_{PE}(a)$ is continuous in $a$. For any $a_1,a_2\in\mathcal{X}$ and any $\lambda\in \Lambda_{PE}(a_1)$,
it is clear that $\lambda-\|a_2-a_1\|\in\Lambda_{PE}(a_2)$.
Hence
$$
\lambda_{PE}(a_2)\ge \lambda_{PE}(a_1)-\|a_2-a_1\|.
$$
Conversely, for any $\lambda\in\Lambda_{PE}(a_2)$, $\lambda-\|a_2-a_1\|\in\Lambda_{PE}(a_1)$. Hence
$$
\lambda_{PE}(a_1)\ge \lambda_{PE}(a_2)-\|a_2-a_1\|.
$$
Therefore,
$$
-\|a_2-a_1\|+\lambda_{PE}(a_2)\le \lambda_{PE}(a_1)\le \|a_2-a_1\|+\lambda_{PE}(a_2)
$$
and $\lambda_{PE}(a)$ is continuous in $a$.

Similarly, it can be proved that $\lambda_{PE}^{'}(a)$ is continuous in $a$.
\end{proof}

\begin{proof} [Proof of Theorem \ref{relation-thm1}]
(1) First, we introduce  the hull $H(a)$ of $a$,
$$
H(a)={\rm cl}\{\sigma_t a(\cdot,\cdot):=a(t+\cdot,\cdot)\,|\, t\in\RR\}
$$
with the open compact topology, where the closure is taken under the open compact topology. Note that, by the almost periodicity of $a(t,x)$ in $t$ uniformly with respect to $x\in\bar D$ (see {\bf (H2)}) and Proposition \ref{almost-periodic-prop}(1),
for any sequence $\{t_n\}\subset \RR$, there is a subsequence $\{t_{n_k}\}$ such that
the limit $\lim_{n_k\to\infty} a(t_{n_k}+t,x)$ exists uniformly in $(t,x)\in\RR\times\bar D$. Hence
 the open compact topology of $H(a)$ is equivalent to the topology of uniform convergence.
Let
\begin{equation}
\label{Phi-t-b-eq}
\Phi(t,b)u_0=u(t,\cdot;b,u_0),
\end{equation}
where $u(t,\cdot;b,u_0)$ is the solution of
\eqref{main-linear-eq}  with $a$ being replaced by $b\in H(a)$  and
$u(0,\cdot;b,u_0)=u_0(\cdot)\in X$.

Note that $(H(a),\sigma_t)$ is a compact minimal flow and  $\nu$ is the unique invariant ergodic measure of  $(H(a),  \sigma_{\tau})$, where   $\nu$ is the Haar measure of $H(a)$.
 It is clear that the map $[0,\infty)\ni t \mapsto \ln\|\Phi(t,b)\|$  is subadditive. By the subadditive ergodic theorem, there are $\lambda_0(a)\in\RR$ and  $H_0(a)\subset H(a)$ with $\nu(H_0(a))=1$ such that $\sigma_t(H_0(a))=H_0(a)$ for any $t\in\RR$ and
\begin{equation}
\label{lyapunov-exp-proof-eq1}
\lim_{t\to\infty}\frac{1}{t}\ln\|\Phi(t,b)\|=\lambda_0(a)
\end{equation}
for any $b\in H_0(a)$.

Next, we prove that  \eqref{lyapunov-exp-proof-eq1} holds for any $b\in H(a)$ and the limit is uniform in $b\in H(a)$.
Assume that this does not hold. Then there are $\epsilon_0>0$, $t_n\to \infty$, and $b_n\in H(a)$ such that
\begin{equation}
\label{assumption-eq1}
|\frac{1}{t_n}\ln \|\Phi(t_n,b_n)\| - \lambda_0(a)|\ge \epsilon_0.
\end{equation}
By the compactness of $H(a)$, there is $b^*\in H(a)$ and a subsequence of $b_n$,  which, without loss of generality, we still denote  as $b_n$,  such that
$$
b_n(t,x)\to b^*(t,x)\quad {\rm as}\quad n\to\infty
$$
uniformly in $t\in\RR$ and $x\in \bar D$. Then
$$
|b^*(t,x)- b_n(t,x)|\le \frac{\epsilon_0}{4}\quad \forall\,\, t\in\RR,\,\, x\in D,\,\, n\gg 1.
$$
Note that $H_0(a)$ is dense in $H(a)$. Therefore there is
$b^{**}\in H_0(a_1)$ such that
\begin{align*}
 |b^{**}(t,x)- b^*(t,x) | \le \frac{\epsilon_0}{4} \quad \forall\, t\in\RR,\,\, x\in D,\,\, n\gg 1.
 \end{align*}
 This implies that
 \begin{align*}
 |b^{**}(t,x)-b_n(t,x)|\le \frac{\epsilon_0}{2}\quad \forall\,\, t\in\RR,\,\, x\in D,\,\,  n\gg 1.
\end{align*}
Then by the comparison principle (see Proposition \ref{comparison-prop}),  we have
\begin{align*}
e^{-\frac{\epsilon_0}{2}t}\Phi(t,b_n)u_0=\Phi(t, b_n - \frac{\epsilon_0}{2})u_0&\le \Phi(t, b^{**})u_0 \\
&\leq \Phi(t, b_n + \frac{\epsilon_0}{2})u_0 =e^{\frac{\epsilon_0}{2}t}\Phi(t,b_n)u_0
\end{align*}
for any $u_0\in X$ with $u_0(x)\ge 0$.
This implies that
\begin{equation}
\label{assumption-eq2}
-\frac{\epsilon_0}{2} t+\ln\|\Phi(t,b_n)\|\le \ln\|\Phi(t,b^{**})\|\le \frac{\epsilon_0}{2} t+\ln\|\Phi(t,b_n)\|\quad \forall\, t\ge 0,\,\, n\gg 1.
\end{equation}
By \eqref{assumption-eq1}  and \eqref{assumption-eq2}, we have
$$
|\lim_{t\to\infty}\frac{1}{t}\ln\|\Phi(t,b^{**})\|- \lambda_0(a)|\ge \frac{\epsilon_0}{2}.
$$
This is a contradiction. Hence \eqref{lyapunov-exp-proof-eq1} holds for any $b\in H(a)$ and the limit is achieved uniformly in $b\in H(a)$.

Now we prove that $\lambda_{PL}(a)=\lambda_{PL}^{'}(a)=\lambda_0(a)$. By the definition of $\Phi(t,s;a)$ (see \eqref{Phi-t-s-eq}) and
$\Phi(t;b)$ (see \eqref{Phi-t-b-eq}), we have
$$
\Phi(t,s;a)=\Phi(t-s;\sigma_s a).
$$
Then, by the above arguments, we have
$$
\lim_{t-s\to\infty}\frac{\ln\|\Phi(t,s;a)\|}{t-s}=\lim_{t-s\to\infty} \frac{\ln\|\Phi(t-s;\sigma_s a)\|}{t-s}.
$$
Hence $\lambda_{PL}(a)=\lambda_{PL}^{'}(a)=\lambda_0(a)$. Moreover, we have
$$
\lambda_{PL}(a)=\lambda_{PL}^{'}(a)=\lim_{t-s\to\infty}\frac{\ln\|\Phi(t,s;a)\|}{t-s}=\lim_{t-s\to\infty}\frac{\ln\|\Phi(t,s;a)u_0\|}{t-s}
$$
for any $u_0\in X$ with $\inf_{x\in D}u_0(x)>0$. This proves (1).

(2)  First, observe  that $ \Phi(t,s;a)$ is exponentially bounded from above as well as from below. That is,  there exist $M,m > 0$ and
$\omega_\pm\in\RR$ such that
$$ me^{\omega_-(t-s)}\le \|\Phi(t,s;a)\| \leq Me^{\omega_+(t-s)}.$$
In fact, let
$$
\mathcal{K}:X \to X, \quad (\mathcal{K}u)(x)=\int_D\kappa(y-x)u(y)dy\quad \forall\, x\in\bar D
$$
and
$$
a_{\min}=\inf_{t\in\RR,x\in\bar D}a(t,x),\,\, a_{\max}=\sup_{t\in\RR,x\in\bar D}a(t,x).
$$
Then we have
$$
e^{a_{\min}(t-s)}e^{\mathcal{K}(t-s)}u_0\le \Phi(t,s)u_0\le e^{a_{\max}(t-s)}e^{\mathcal{K}(t-s)}u_0
$$
for all $t\ge s$ and $u_0\in X$ with $u_0\ge 0$.
Note that
$$
u_0\le e^{\mathcal{K}(t-s)}u_0\le e^{\|\mathcal{K}\|(t-s)}\|u_0\|
$$
for any $t\ge s$ and $u_0\in X$ with $u_0\ge 0$. It then follows that
$$
e^{a_{\min}(t-s)}\le \|\Phi(t,s;a)\|\le e^{(a_{\max}+\|\mathcal{K}\|)(t-s)}\quad \forall\, t\ge s.
$$
Therefore $ \Phi(t,s;a)$ is exponentially bounded from above as well as from below.

Next, we prove that $\lambda_{PL}(a)\le \lambda_{PD}(a)$.
To this end, for any given $\epsilon > 0$, let $\lambda_* = \lambda_{PD} (a)+ \epsilon$.  Then we can find $M > 0$ such that;
$$\|\Phi_{\lambda_*}(t,s;a)\| = \|e^{-\lambda_*(t-s)}\Phi(t,s;a)\| \leq M\quad \forall \, t\ge s.$$
That is
$$\|\Phi(t,s;a)\| \leq Me^{\lambda_*(t-s)}\quad \forall\, t \geq s.$$
It then follows that,
$$\limsup_{t-s\to\infty}\disp\frac{\ln \|\Phi(t,s;a)\|}{t-s} \leq \lambda_*,$$
which implies $\lambda_{PL}(a) \leq \lambda_{PD}(a) + \epsilon.$
Letting $\epsilon\to 0$.
we conclude that $\lambda_{PL}(a) \leq \lambda_{PD}(a).$

Now, we prove that $\lambda_{PD}(a)\le \lambda_{PL}(a)$. To this end, for any $\epsilon > 0$, let
 $\bar\lambda = \lambda_{PL}(a) + \epsilon$. We have
$$ \|\Phi_{\bar\lambda}(t,s;a)\| = e^{-(\lambda_{PL}(a) + \epsilon)(t-s)}\|\Phi(t,s;a)\| \rightarrow 0$$
 as $t-s \rightarrow \infty.$
This implies that $\Phi_{\lambda_{PL} (a)+ \epsilon}(t,s;a)$ admits an exponential dichotomy with $P = 0.$
So $\lambda_{PL} (a)+ \epsilon \in \mathbb{R}\setminus\Sigma(a) $, and then  $\lambda_{PD} (a)\leq \lambda_{PL}(a) + \epsilon$.  Since $\epsilon >0$ is arbitrary,  we conclude that $\lambda_{PD} (a)\leq \lambda_{PL}(a).$
Hence $\lambda_{PL}(a)=\lambda_{PD}(a)$.
\end{proof}

\section{Proof of Theorem \ref{relation-thm2}}

In this section, we discuss the relations between $\lambda_{PE}(a)$, $\lambda_{PE}^{'}(a)$
and $\lambda_{PL}(a)$ and prove  Theorem \ref{relation-thm2}.

We first present four lemmas.

\begin{lemma}
\label{mean-lm2}
For any $x\in D$ and  $\epsilon>0$, there is $A_{x,\epsilon}\in W^{1,\infty}(\RR)$  such   that
$$
a(t,x)+A_{x,\epsilon}^{'}(t)\ge \hat a(x)-\epsilon \quad \text{for} \;\; a.e.\,\, t\in\RR.
$$
\end{lemma}

\begin{proof}
It follows from \cite[Lemma 3.2]{NadRos}.
\end{proof}

\begin{lemma}
\label{monotonicity-lem}
If $D_1\subset D_2$, then $\lambda_{PL}(D_1)\le \lambda_{PL}(D_2)$.
\end{lemma}

\begin{proof}
For $u_0(x)\equiv 1$ on $D_2$, we have
$$
\Phi(t,s;a,D_1)u_0|_{D_1}\le\Phi(t,s;a,D_2)u_0\quad {\rm on}\quad D_1,\,\,\, \forall\, t\ge s.
$$
This implies that
$$
\lambda_{PL}(a,D_1)=\lim_{t-s\to\infty}\frac{\ln |\Phi(t,s;a,D_1)u_0|_{D_1}\|}{t-s}\le\lim_{t-s\to\infty}\frac{\ln\|\Phi(t,s;a,D_1)u_0\|}{t-s}=\lambda_{PL}(a,D_2).
$$
\end{proof}

\begin{lemma}
\label{lower-bound-lem}
$\lambda_{PL}(a)\ge \underset{{x\in D}}{\sup}\;\hat a(x)$.
\end{lemma}

\begin{proof}
Note that this lemma follows from $\lambda_{PL}(a)\ge \lambda_{PE}(a)$ (see Theorem \ref{relation-thm2}(2)) and $\lambda_{PE}(a)\ge \sup_{x \in D}\hat a(x)$ (see Theorem \ref{variation-thm1}(1)),
whose proofs are independent of each other and do not require the conclusion in this lemma. In the following,
we give a direct proof of this lemma.

For any $\epsilon>0$,
let $x_0\in D$ be such that
$$
\hat a(x_0)\ge \sup_{x\in D}\hat a(x)-\epsilon.
$$
By Lemma \ref{mean-lm2},  there are $\delta>0$  and $A_0\in W^{1,\infty} (\RR)$ such that
\begin{equation}
\label{eigenvalue-eq1-0}
a(t,x_0)+A_0^{'}(t)\ge \hat a(x_0)-\epsilon\quad \text{for}\;\; a.e. \,\,  t\in\RR
\end{equation}
and
\begin{equation}
\label{eigenvalue-eq1-00}
a(t,x)\ge  a(t,x_0)-\epsilon\quad \forall\, t\in\RR,\,\, x\in D_1,
\end{equation}
where
$$
D_1=D_1(x_0,\delta)=\{x\in D\,|\, |x-x_0|\le \delta\}.
$$

Let $u(t,x;D_1)$ be the solution of
$$
u_t=\int_{D_1}\kappa(y-x)u(t,y)dy+a(t,x)u,\quad x\in \bar D_1
$$
with $u(0,x;D_1)=1$. Let $v(t,x;D_1)=e^{A_0(t)}u(t,x;D_1)$. Then
$$
v_t=\int_{D_1}\kappa(y-x) v(t,y;D_1)dy+(a(t,x)+A_0^{'}(t))v(t,x;D_1)\quad \text{for} \;\; a.e.\, \, t\ge 0,\,\, \forall\, x\in\bar D_1.
$$
This together with Proposition \ref{comparison-prop}, \eqref{eigenvalue-eq1-0},  and \eqref{eigenvalue-eq1-00} implies that
$$
v(t,x;D_1)\ge e^{A_0(0)} e^{(\hat a(x_0)-2\epsilon)t}\quad  for \; a.e. \,\,  t\ge 0,\,\, \forall\, x\in D_1.
$$
Hence
$$
\lambda_{PL}(D_1)\ge \hat a(x_0)-2\epsilon\ge \sup_{x\in D}\hat a(x)-3\epsilon.
$$
By Lemma \ref{monotonicity-lem}, we have
$$
\lambda_{PL}(D)\ge \sup_{x\in D}\hat a(x)-3\epsilon
$$
for any $\epsilon>0$. Letting $\epsilon\to 0$, the lemma follows.
\end{proof}

Let $a(t,x), g(\cdot,\cdot)\in \mathcal{X}$ and $a(t,x)$ be almost periodic in $t$ uniformly with respect to $x\in\bar D$.
Consider
\begin{equation}
\label{nonhomogeneous-eq}
\frac{d\phi}{dt}=a(t,x)\phi(t)-\lambda \phi(t)+g(t,x),
\end{equation}
where $\lambda\in\RR$ is a constant and $x\in\bar D$.
\eqref{nonhomogeneous-eq} can be viewed as a family of ODEs with parameter $x\in\bar D$.

\begin{lemma}
\label{nonhomogeneous-lm1}
If $\lambda>\sup_{x\in D}\hat a(x)$, then for any $x\in \bar D$,
$$
\phi^*(t;x,g)=\int_{-\infty}^t e^{\int_s^t a(\tau,x)d\tau-\lambda(t-s)} g(s,x)ds
$$
is a unique bounded solution of \eqref{nonhomogeneous-eq} on $\RR$. Moreover,
$\phi^*(t;x,g)$ is uniformly continuous in $(t,x)\in\mathbb{R} \times \bar D$. If
$\inf_{t\in\mathbb{R},x\in\bar D}g(t,x)>0$, then
$\inf_{t\in\mathbb{R},x\in\bar D}\phi^*(t;x,g)>0$.
\end{lemma}

\begin{proof}
First, since $\lambda>\sup_{x\in D}\hat a(x)$, it is not difficult to prove that \eqref{nonhomogeneous-eq} has at most one bounded solution.
Note that there is $\delta>0$ such that
$$
e^{\int_s ^t a(\tau,x)d\tau -\lambda(t-s)}\le e^{-\delta(t-s)}\quad \forall\, t>s,\,\, x\in\bar D.
$$
This implies that $\phi^*(t;x,g)$ is uniformly bounded in $t\in\RR$ and $x\in\bar D$. Moreover,
 by direct computation, we have that $\phi^*(t;x,g)$ is a bounded  solution of \eqref{nonhomogeneous-eq} on $\RR$ and then
 $\frac{d\phi^*}{dt}(t;x,g)$ is uniformly bounded.
 Hence $\phi^*(t;x,g)$ is uniformly continuous in $t$ uniformly  with respect to $x\in\bar D$.

 Next, we claim that $\phi^*(t;x,g)$ is uniformly continuous in $x\in\bar D$ uniformly with respect to $t\in\RR$.
 In fact, if the claim is not true, then there are $\epsilon_0>0$, $x_n,\tilde x_n\in\bar D$, and $t_n\in\RR$ such that
 $$
 |x_n-\tilde x_n|\le \frac{1}{n}\quad \forall\, n\ge 1
 $$
 and
 \begin{equation}
 \label{assumption-eq3}
 |\phi^*(t_n;x_n,g)-\phi^*(t_n;\tilde x_n,g)|\ge \epsilon_0\quad \forall\, n\ge 1.
 \end{equation}
 Let
 $$
 \phi_n(t)=\phi^*(t+t_n;x_n,g),\quad \tilde\phi_n(t)=\phi^*(t+t_n;\tilde x_n,g).
 $$
 Then $\phi_n(t)$ and $\tilde\phi_n(t)$ satisfy
 $$
 \phi_n^{'}(t)=a(t+t_n,x_n)\phi_n(t)-\lambda \phi_n(t)+g(t+t_n,x_n)
 $$
 and
 $$
 \tilde\phi_n^{'}(t)=a(t+t_n,\tilde x_n)\phi_n(t)-\lambda \phi_n(t)+g(t+t_n,\tilde x_n),
 $$
 respectively. Without loss of generality, we may assume that there are $b(t)$, $h(t)$, $\phi(t)$, and $\tilde \phi(t)$ such that
 $$
 \lim_{n\to\infty} a(t+t_n,x_n)=\lim_{n\to\infty} a(t+t_n,\tilde x_n)=b(t),\quad \lim_{n\to\infty} g(t+t_n,x_n)=\lim_{n\to\infty} g(t+t_n,\tilde x_n)=h(t),
 $$
 and
 $$ \lim_{n\to\infty}\phi_n(t)=\phi(t),\quad
 \lim_{n\to\infty}\tilde\phi_n(t)=\tilde\phi(t)
 $$
 locally uniformly in $t\in\RR$. It then follows that both $\phi(t)$ and $\tilde\phi(t)$ are bounded solutions of the following
 ODE
 $$
 \psi^{'}=b(t)\psi-\lambda \psi+h(t).
 $$
 Since $\lambda>\sup_{t\in\RR} b(t)$, this ODE has a unique bounded solution. This implies that
 $$\phi(t)\equiv \tilde \phi(t).
 $$
 But   by \eqref{assumption-eq3},
 $$
 |\phi(0)-\tilde\phi(0)|\ge \epsilon_0,
 $$
 which is a contradiction. Therefore, the claim holds, whence $\phi^*(t;x,g)$ is uniformly continuous in
 $t\in\RR$ and $x\in\bar D$.

 We now claim that, if
$g_{\inf}:=\inf_{t\in\mathbb{R},x\in\bar D}g(t,x)>0$, then
$\inf_{t\in\mathbb{R},x\in\bar D}\phi^*(t;x,g)>0$.
 In fact,  let $a_{\inf}=\inf_{t\in\RR,x\in\bar D}a(t,x)$. For any $t\in\RR$ and $x\in\bar D$, we have
 \begin{align*}
\phi^*(t;x,g)&=\int_{-\infty}^t e^{\int_s^t a(\tau,x)d\tau-\lambda(t-s)} g(s,x)ds\\
&\ge \int_{-\infty}^t e^{(a_{\inf}-\lambda)(t-s)} g_{\inf} ds\\
&=\frac{g_{\inf}}{\lambda-a_{\inf}}.
\end{align*}
 The claim then follows and the lemma is thus proved.
\end{proof}

Now, we present the proof of Theorem \ref{relation-thm2}.

\begin{proof}[Proof of Theorem \ref{relation-thm2}(1)] The proof is given in two steps.

\smallskip

\noindent  {\bf Step 1.}  In this step, we prove that $\lambda_{PE}^{'}(a)\le\lambda_{PL}(a)$.

 Note that, for any $\lambda>\lambda_{PL}(a)$, there are $M,\delta >0$ such that
\begin{equation}\label{expon bound}
 e^{-\lambda(t-s)}\|\Phi(t,s;a)\|\le M e^{-\delta(t-s)}\quad \forall\, t\ge s.
\end{equation}
 For given $v\in\mathcal{X}$, consider
\begin{equation}
\label{aux-aux-eq1}
u_t=\int_{D}\kappa(y-x)u(t,y)dy+a(t,x)u-\lambda u +v.
\end{equation}
 Recall that
$$
\Phi_\lambda(t,s;a)=e^{-\lambda(t-s)}\Phi(t,s;a).
$$
Let
\begin{equation}
\label{u-eq00}
u(t,\cdot;a,v)=\int_{-\infty}^t \Phi_\lambda (t,s;a)v(s,\cdot)ds.
\end{equation}
By direct computation, we have that $u(t,x;a,v)$ is a solution of \eqref{aux-aux-eq1}.  By \eqref{expon bound},
 we have that $u(t,x;a,v)$ is bounded, and then by \eqref{aux-aux-eq1}, $u(t,x;a,v)$ is uniformly continuous in $t$ uniformly with respect to $x\in\bar D$.

  Let $g(t,x)=\int_D \kappa(y-x)u(t,y;a,v)dy+v(t,x)$. We have $g\in\mathcal{X}$.
  By Lemma \ref{lower-bound-lem}, $\lambda>\sup_{x\in D}\hat a(x)$.  Then  by Lemma \ref{nonhomogeneous-lm1},
  $u(t,x;a,v)=\phi^*(t;x,g)$ and then $u(\cdot,\cdot;a,v)\in\mathcal{X}$.
Choose $v(t,x)\equiv 1$. By Lemma \ref{nonhomogeneous-lm1} again, we have  $\inf_{t\in\RR,x\in\bar D}u(t,x;a,v)>0$. Note that
$$
-u_t+\int_D\kappa(y-x)u(t,y;a,v)dy+a(t,x)u(t,x;a,v)=\lambda u(t,x;a,v)-v\le \lambda u(t,x;a,v).
$$
Hence
 $\lambda\in\Lambda_{PE}^{'}(a)$. Therefore,
$$
\lambda_{PE}^{'}(a)\le \lambda\quad \forall \, \lambda>\lambda_{PL}(a).
$$
This implies that
$$
\lambda_{PE}^{'}(a)\le \lambda_{PL}(a).
$$

\noindent {\bf Step 2.}  In this step,  we prove that $\lambda_{PE}^{'}(a)\ge \lambda_{PL}(a)$.

 Note that for any $\lambda >\lambda_{PE}^{'}(a)$, there is $\phi\in \mathcal{X}$ with
 $\inf_{t\in\RR,x\in\bar D}\phi(t,x)>0$  such that
$$
-\phi_t(t,x) +\int_D \kappa(y-x)\phi(t,y)dy+a(t,x)\phi(t,x)\le \lambda \phi(t,x)\quad a.e.\,  t\in\RR,\,\,\forall\,  x\in \bar D.
$$
Let $u_0=\inf_{t\in\RR,x\in\bar D} \phi(t,x)$. By Proposition \ref{comparison-prop}, we have
$$
\Phi(t,0;a)u_0\le e^{\lambda t}\phi(t,x)\quad \forall \,\, t\ge 0,\, \, x\in\bar D.
$$
This implies that
$$
\lambda_{PL}(a)\le \liminf_{t\to\infty}\frac{\ln\|\Phi(t,0;a)u_0\|}{t}\le \lambda.
$$
Hence $\lambda_{PL}(a)\le \lambda_{PE}^{'}(a)$ and then  $\lambda_{PE}^{'}(a)= \lambda_{PL}(a)$.
\end{proof}

\begin{proof}[Proof of Theorem \ref{relation-thm2}(2)]
We prove Theorem \ref{relation-thm2}(2) in three steps.

\smallskip

\noindent {\bf Step 1.}
 In this first step, we prove that $\lambda_{PE}(a)\le \lambda_{PL}(a)$ for any domain $D$.

  Choose any
  $\lambda \in \Lambda_{PE}$. There is $\phi\in\mathcal{X}$ with $\inf_{t\in\RR}\phi(t,x) \ge\not\equiv 0$ and
   $\lambda \phi \leq L\phi$.
    Set $w(t,x) = e^{\lambda t}\phi(t,x).$ Then $w(t,x)$ is  a subsolution of \eqref{main-linear-eq} and $w(0,x) = \phi(0,x)$. By comparison principle,
 we have
 $$
 e^{\lambda t} \phi(t,\cdot) \le  \Phi(t,0;a)w(0,\cdot)\quad \forall \, t\ge 0.
 $$
 This implies that  $\lambda \leq \lambda_{PL}(a)$.  Hence
 \begin{equation}\label{lambdap}
      \lambda_{PE}(a) \leq \lambda_{PL}(a).
 \end{equation}

 \noindent {\bf Step 2.} In this step, we assume that $a(t,x)$ is $T$-periodic in $t$  and is also periodic in $x$ if $D=\RR^N$,
 and prove $\lambda_{PE}(a)=\lambda_{PL}(a)$.

By Proposition \ref{periodic-prop},  for any $\epsilon>0$, there are  $a_\epsilon(t,x), \phi_\epsilon(t,x)\in \mathcal{X}_p$ such that $\phi_\epsilon(t,x)>0$,
 $$
 \|a-a_\epsilon\|<\epsilon,
 $$
 and
 $$
 -\p_t \phi_\epsilon(t,x)+\int_D \kappa(y-x)\phi_\epsilon(t,y)dy+a_\epsilon(t,x)\phi_\epsilon(t,x)=\lambda_{PL}(a_\epsilon)\phi_\epsilon(t,x).
 $$
 This implies that
 $$
 \lambda_{PL}(a_\epsilon)-\|a-a_\epsilon\|\in\Lambda_{PE}(a).
 $$
 It then follows that
 $$
 \lambda_{PE}(a)\ge \lambda_{PL}(a_\epsilon)-\|a-a_\epsilon\|\ge \lambda_{PL}(a)-2\epsilon.
 $$
 Letting $\epsilon\to 0$, we get $\lambda_{PE}(a)\ge \lambda_{PL}(a)$, which together with \eqref{lambdap} implies that $\lambda_{PE}(a)=\lambda_{PL}(a)$.

 \smallskip

 \noindent{\bf Step 3.} In this step, we assume that $a(t,x)$ is limiting almost periodic and  prove that $\lambda_{PE}(a)=\lambda_{PL}(a)$.

 Since $a(t,x)$ is limiting almost periodic, there is a sequence  $\{a_n(t,x)\}$ of periodic functions such that
 $$
 \lim_{n\to\infty} a_n(t,x)=a(t,x)
 $$
 uniformly in $t\in\RR$ and $x\in\bar D$. Then by   Lemma \ref{continuity-lem1} and  the arguments in {\bf Step 2},
 $$
 \lambda_{PE}(a)=\lim_{n\to\infty}\lambda_{PE}(a_n)=\lim_{n\to\infty}\lambda_{PL}(a_n)=\lambda_{PL}(a).
 $$
 The proof of Theorem \ref{relation-thm2}(2) is thus completed.
\end{proof}

\begin{proof}[Proof of Theorem \ref{relation-thm2}(3)]
Assume that $a(t,x)\equiv a(t)$.

 First, we prove  that for any $D$,
 \begin{equation}
 \label{space-homogeneous-eq1}
 \lambda_{PL}(a)=\hat a+\lambda_{PL}(0).
 \end{equation}
Note that
$$
\Phi(t;a)=e^{\int_0^t a(s)ds}\Phi(t;0).
$$
This implies that \eqref{space-homogeneous-eq1} holds.

 Next, we prove that for any $D$,
 \begin{equation}
 \label{space-homogeneous-eq2}
 \lambda_{PE}(a)=\hat a+\lambda_{PE}(0).
 \end{equation}
 To this end, we first consider the case that ${\int_0^t a(s)ds}-\hat a t$
is a bounded function of $t$.
We claim that $\Lambda_{PE}(a)=\Lambda_{PE}(\hat a)$.
In fact, for any $\lambda\in\Lambda_{PE}(a)$, let $\phi\in\mathcal{X}$ be such that $\inf_{t\in\RR}\phi(t,x)\ge \not\equiv 0$ and
$$
-\phi_t+\int_D\kappa(y-x)\phi(t,y)dy+a(t)\phi(t,x)\ge \lambda \phi(t,x).
$$
Let $\psi(t,x)=e^{-(\int_0^t a(s)ds-\hat a t)}\phi(t,x)$. Then $\psi\in\mathcal{X}$, $\inf_{t\in\RR}\psi(t,x)\ge \not\equiv 0$, and
\begin{align*}
-\psi_t (t,x)&= (a(t)-\hat a) \psi(t,x)-e^{-(\int_0^t a(s)ds-\hat a t)}\phi_t(t,x)\\
&\ge (a(t)-\hat a) \psi(t,x) +e^{-(\int_0^t a(s)ds-\hat a t)} \Big(-\int_D\kappa(y-x)\phi(t,y)dy-a(t)\phi(t,x)+\lambda\phi(t,x)\Big)\\
&=-\hat a \psi(t,x)-\int_D \kappa(y-x)\psi(t,y)dy +\lambda\psi(t,x).
\end{align*}
This implies that $\lambda\in \Lambda_{PE}(\hat a)$.

Conversely, for any $\lambda \in\Lambda_{PE}(\hat a)$, there is $\phi\in\mathcal{X}$ with $\inf_{t\in\RR}\phi(t,x)\ge\not\equiv 0$ such that
$$
-\phi_t+\int_D\kappa(y-x)\phi(t,y)dy+\hat{a} \phi(t,x)\ge \lambda \phi(t,x).
$$
Let $\psi(t,x)=e^{-(\hat a t-\int_0^t a(s)ds)}\phi(t,x)$. Then $\psi\in\mathcal{X}$, $\inf_{t\in\RR}\psi(t,x)\ge\not\equiv 0$, and
$$
-\psi_t(t,x)\ge -a(t,x)\phi(t,x)-\int_D \kappa(y-x)\psi(t,y)dy+\lambda \psi(t,x).
$$
This implies that $\lambda\in\Lambda_{PE}(a)$.
Therefore, $\Lambda_{PE}(a)=\Lambda_{PE}(\hat a)$ and then $\lambda_{PE}(a)=\lambda_{PE}(\hat a)=\hat a+\lambda_{PE}(0)$.
\eqref{space-homogeneous-eq2} follows.

We now  consider the general case. Let $a(t)$ be any given almost periodic function. By Proposition \ref{almost-periodic-prop}(3),  we have that for any $\epsilon>0$, there is an almost periodic function $ a_\epsilon(t)$ such that
$\int_0^t a_\epsilon(s)ds-\hat a_\epsilon t$ is bounded and
$$
\|a(\cdot)-a_\epsilon(\cdot)\|\le \epsilon.
$$
By the above arguments, $\lambda_{PE}(a_\epsilon)=\hat a_\epsilon+\lambda_{PE}(0)$.
By Lemma \ref{continuity-lem1},
$$
\hat a+\lambda_{PE}(0)-2\epsilon \le \lambda_{PE}(a)\le \hat a+\lambda_{PE}(0)+2\epsilon
$$
Letting $\epsilon\to 0$,  \eqref{space-homogeneous-eq2} follows.

Now, by  similar arguments, we have that for any $D$,
\begin{equation}
\label{space-homogeneous-eq3}
\lambda_{PE}^{'}(a)=\hat a+\lambda_{PE}^{'}(0).
\end{equation}.

Finally,  by (1), (2), $\lambda_{PL}(0)=\lambda_{PE}(0)=\lambda_{PE}^{'}(0)$. This together with
\eqref{space-homogeneous-eq1}, \eqref{space-homogeneous-eq2}, and \eqref{space-homogeneous-eq3}
implies (3).
 \end{proof}

\section{Proof of Theorem \ref{variation-thm1}}

In this section, we discuss the effects of time and space variations on $\lambda_{PE}(a)$ and
prove Theorem \ref{variation-thm1}. We first present a lemma.

\begin{lemma}
\label{eigenvalue-lm1}
Consider \eqref{main-nonlinear-eq}. Suppose that
$f(t,x,u)=u(a(x)-b(x))$, $a,b\in X$,   and $\inf_{x\in D}b(x)>0$.
If   $\lambda_{PE}(a,D_0)>0$ for some bounded subset $D_0\subset D$, then \eqref{main-nonlinear-eq}
 has a  positive stationary solution $\phi^*(\cdot)\in X$.
\end{lemma}

\begin{proof} Let $D_1\subset D_2\subset \cdots \subset D_n\subset \cdots$ be a sequence of bounded domains
such that $D=\cup_{n=1}^\infty D_n$.   Then by Theorem \ref{relation-thm2} and Lemma \ref{monotonicity-lem},
$$
\lambda_{PE}(a,D_n)=\lambda_{PL}(a,D_n)\le \lambda_{PL}(a,D_{n+1})= \lambda_{PE}(a,D_{n+1})\quad \forall \, n\ge 1.
$$
By Proposition \ref{periodic-prop},
$$
\lambda_{PE}(a,D_n)=\lambda_{PL}(a,D_n)\ge \lambda_{PL}(a,D_0)>0\quad \forall\, n\gg 1.
$$
Then by \cite[Theorem E]{RaSh}, there is a unique positive stationary solution
$\phi^*_n(\cdot)\in X(D_n)$ of
$$
u_t=\int_{D_n}\kappa(y-x)u(t,y)dy+u(a(x)-b(x)u),\quad x\in \bar D_n
$$
for $n\gg 1$. By Proposition \ref{comparison-prop},
$$
\phi_n^*(x)\le \phi_{n+1}^*(x)\quad \forall\,\, x\in D_n,\,\, n\gg 1.
$$
Therefore, the limit $\phi^*(x)=\lim_{n\to\infty}\phi_n^*(x)$ exists for all $x\in \bar D$.  Moreover, it is not difficult to see that $u=\phi^*(x)$ is a positive stationary solution of \eqref{main-nonlinear-eq}.
\end{proof}

We now prove Theorem \ref{variation-thm1}.

 \begin{proof}[Proof of Theorem \ref{variation-thm1}]
(1). We first prove that for any $a$ satisfying {\bf (H2)}, $\lambda_{PE}(a) \ge \sup_{x\in D} \hat{a}(x).$
For any $\epsilon>0$,
let $x_0\in D$ be such that
$$
\hat a(x_0)\ge \sup_{x\in D}\hat a(x)-\epsilon.
$$
By Lemma \ref{mean-lm2},  there are $\delta>0$  and $A_0\in W^{1,\infty} (\RR)$ such that
\begin{equation}
\label{eigenvalue-eq1}
a(t,x_0)+A_0^{'}(t)\ge \hat a(x_0)-\epsilon\quad {\rm for}\,\, \,\, a.e. \, \,  t\in\RR
\end{equation}
and
\begin{equation}
\label{pl-pe-eq3}
a(t,x)\ge  a(t,x_0)-\epsilon\quad \forall\, t\in\RR,\,\, x\in D_1(x_0,\delta),
\end{equation}
where
$$
D_1(x_0,\delta)=\{x\in D\,|\, |x-x_0|\le \delta\}.
$$

By \eqref{eigenvalue-eq1}, there is  $\tilde a(\cdot)\in X$  such that
$$
\tilde a(x)\begin{cases} =\hat a(x_0)-\epsilon\quad x\in D_1(x_0,\delta/2)\cr
\le a(t,x)+A_0^{'}(t)\quad {\rm for}\,\, a.e.\,\,  t\in\RR,\,\,\forall\,  x\in D.
\end{cases}
$$
For any $\lambda<\hat a(x_0)-\epsilon$, consider
\begin{equation}
\label{eigenvalue-eq2}
\tilde u_t=\int_{D}\kappa(y-x)\tilde u(t,y)dy+\tilde u(t,x)(\tilde a(x)-\lambda-A_0^{'}(t)-e^{A_0(t)}\tilde u),\quad x\in D.
\end{equation}
Let $\tilde v(t,x)=e^{A_0(t)}\tilde u(t,x)$. Then $\tilde v(t,x)$ satisfies
\begin{equation}
\label{eigenvalue-eq3}
\tilde v_t=\int_{D_1}\kappa(y-x)\tilde v(t,y)dy+\tilde v(t,x)(\tilde a(x)-\lambda-\tilde v),\quad{\rm for}\,\, a.e.\, t\in \RR,\, \, \forall\,  x\in D.
\end{equation}
By Lemma \ref{eigenvalue-lm1}, there is $\tilde v^*\in X$ with $\tilde v^*(x)>0$ such that
$$
 \int_{D}\kappa(y-x)\tilde v^*(y)dy+\tilde v^*(x)(\tilde a(x)-\lambda-\tilde v^*(x)) = 0\quad  \forall\, x\in D.
$$
Let $\tilde u^*(t,x)=\tilde v^*(x)e^{-A_0(t)}$. We have
$$
-\tilde u^*_t+\int_{D}\kappa(y-x)\tilde u^*(t,y)dy+(\tilde a(x)-A_0^{'})\tilde u^*(t,x)\ge \lambda  \tilde u^*(t,x)
$$
for a.e. $t\in\RR$ and all $x\in D$.
This implies that
$$
-\tilde u^*_t+\int_{D}\kappa(y-x)\tilde u^*(t,y)dy+a(t,x)\tilde u^*(t,x)\ge \lambda \tilde u^*(t,x)
$$
for a.e. $t\in\RR$ and all $x\in D$. Hence $\lambda\in \Lambda_{PE}(a)$, and
$$
\lambda_{PE}(a)\ge \sup_{x\in D}\hat a(x)-2\epsilon.
$$
Letting $\epsilon\to 0$, we obtain that $\lambda_{PE}(a)\ge \sup_{x\in D}\hat a(x)$.

Next, we assume that $a(t,x)$ is limiting almost periodic and show that $\lambda_{PE}(a) \ge \lambda_{PE}(\hat{a}) \ge \sup_{x \in D}\hat{a}(x)$.
Let $a_n(t,x)$ be a sequence of periodic functions such that $\underset{n\to \infty}{\lim}a_n(t,x) = a(t,x)$ uniformly in $t \in \RR$ and $x \in \bar{D}$.
By Theorem \ref{relation-thm2}(2) and Proposition \ref{periodic-prop}(1), (3),
we have
$$ \lambda_{PE}(a_n) \ge \lambda_{PE}(\hat{a}_n) \ge \sup_{x \in D}\hat{a}_n(x).$$
Letting $n\to \infty$, by Lemma \ref{continuity-lem1}, we obtain
$$
\lambda_{PE}(a) \ge \lambda_{PE}(\hat{a})  \ge \sup_{x \in D}(\hat{a}(x).
$$
(1) is thus proved.

\smallskip

(2) Write the eigenvalue problem
$$
\int_D\kappa(y-x)\phi(y)dy+a(x)\phi(x)=\lambda\phi(x)\quad\forall\,  x\in\bar D
$$
as
$$
\int_D\kappa(y-x)[\phi(y)-\phi(x)]dy+ [a(x)+\int_D \kappa(y-x)dy]\phi(x)=\lambda \phi(x)\quad \forall \, x\in \bar D.
$$
Then by the arguments of \cite[Theorem 2.1(4)]{ShXi},
$$
\lambda_{PE}(a)\ge \bar a+\frac{1}{|D|}\int_D\int_D \kappa(y-x)dydx.
$$

(3) Let $R_n\to \infty$ and $B(0,R_n)=\{x\in\RR^N\,|\, \|x\|\le R_n\}$. Then by Theorem \ref{relation-thm2} and Lemma \ref{monotonicity-lem},
 $$
\lambda_{PE}(a, B(0,R_n))=\lambda_{PL}(a,B(0,R_n))\le \lambda_{PL}(a,B(0,R_{n+1}))=\lambda_{PE}(a,B(0,R_{n+1}))\quad \forall\, n\ge 1.
$$
Put
$$
\lambda_{\infty}(a,D)=\lim_{n\to\infty} \lambda_{PE}(a,B(0,R_n))>0.
$$
Then for any $\lambda<\lambda_\infty(a,D)$, 
$$
\lambda(a,B(0,R_n))-\lambda>0\quad \forall\, n\gg 1.
$$
By  Lemma \ref{eigenvalue-lm1}, there is $\phi\in X^+\setminus\{0\}$ such that
$$
\int_D \kappa(y-x)\phi(y)dy+a(x)\phi(x)=\lambda \phi(x)+\phi^2(x)\ge \lambda \phi(x)\quad \forall\, x\in D.
$$
This implies that
$$
\lambda_{PE}(a,D)\ge \lambda \quad \forall\, \lambda <\lambda_{\infty}(a,D)
$$
and hence
\begin{equation}
\label{new-eqq2}
\lambda_{PE}(a,D)\ge \lambda_{PE}(a,B(0,R_n))\quad \forall\, n\ge 1.
\end{equation}

By (2), we have
$$
\lambda_{PE}(a,B(0,R_n))\ge \frac{1}{|B(0,R_n)|}\int_{B(0,R_n)}a(x)dx+\frac{1}{|B(0,R_n)|} \int_{B(0,R_n)}\int_{B(0,R_n)}\kappa(y-x)dydx.
$$
By {\bf (H1)}, for any $\epsilon>0$, there is $r>0$ such that
$$
\int_{\RR^N\setminus B(0,r)}\kappa(z)dz<\epsilon.
$$
This implies that
\begin{align*}
\int_{B(0,R_n)}\int_{B(0,R_n)}\kappa(y-x)dydx &\ge \int_{B(0,R_n-r)}\int_{B(0,R_n)}\kappa (y-x)dydx\\
&\ge \int_{B(0,R_n-r)}[\int_{\RR^N}\kappa(y-x)dy-\epsilon]dx\\
&=\int_{B(0,R_n-r)}(1-\epsilon)dx=|B(0,R_n-r)| (1-\epsilon).
\end{align*}
Note that
$$
\frac{|B(0,R_n-r)|}{|B(0,R_n)|}=\frac{(R_n-r)^N}{R_n^N}\to 1\quad {\rm as}\quad n\to\infty.
$$
It then follows that
$$
\lambda_{PE}(a)\ge \hat a+1-\epsilon\quad \forall\, \epsilon>0.
$$
Let $\epsilon\to 0$, we have
$$
\lambda_{PE}(a)\ge \hat a+1.
$$
The theorem is thus proved.
\end{proof}

 \section{Proof of Theorem \ref{variation-thm2}}

 In this section, we discuss the effects of space and time variations on $\lambda_{PL}(a)$ and
 prove Theorem \ref{variation-thm2}.
 We first present a lemma.

\begin{lemma}
\label{lyapunov-exp-lm2}
For any given $T>0$ and compact subset $\Omega\subset\RR^N$, let $w(t,x)$ be a positive continuous function on $[0,T]\times \Omega$.
Let
$$
\theta(x,y)=\frac{1}{T}\int_0^T \frac{w(t,y)}{w(t,x)}dt.
$$
Then either $w(t,x)$ is independent of $x$ or there is $x^*\in\Omega$ such that
$$
\theta(x^*,y)\ge 1\quad \forall\, y\in\Omega
$$
with strictly inequality for some $y\in\Omega$.
\end{lemma}

\begin{proof}
It follows from \cite[Lemma 4.3]{HuShVi}.
\end{proof}

We now prove Theorem \ref{variation-thm2}.

\begin{proof}[Proof of Theorem \ref{variation-thm2}]
(1)  First we assume that $D$ is bounded.
Let $u_0^*(x)\equiv 1$.
Let $u(t,\cdot;u_0^*)=\Phi(t;a)u_0^*$ and
$$v(t,\cdot;u_0^*)=e^{-\lambda_{PL}(a)t}u(t,\cdot;u_0^*).
$$
Then
$$
\limsup_{t\to\infty}\frac{\ln \|v(t,\cdot;u_0^*)\|}{t}=0
$$
and $v(t,x;u_0^*)$ satisfies
$$
\lambda_{PL}(a)v=-v_t+\int_{D}\kappa(y-x)v(t,y)dy+a(t,x)v(t,x)\quad \forall\, t\ge 0,\,\, x\in D.
$$
Hence
\begin{equation}
\label{lyapunov-exp-eq0}
\lambda_{PL}(a)=-\frac{v_t(t,x;u_0^*)}{v(t,x;u_0^*)}+\int_{D}\kappa(y-x)\frac{v(t,y;u_0^*)}{v(t,x;u_0^*)}dy+a(t,x)\quad \forall\, t\ge 0,\,\, x\in D.
\end{equation}

For any $\epsilon>0$, by Proposition \ref{periodic-prop},  there are $a^*\in X$ and $\phi^*\in X$ with $\phi^*(x)>0$ such that
\begin{equation}
\label{lyapunov-exp-eq1}
a^*(x)\le \hat a(x)\le a^*(x)+\epsilon,
\end{equation}
\begin{equation}
\label{lyapunov-exp-eq2}
\lambda_{PL}(\hat a)-\epsilon\le \lambda_{PL}(a^*)\le \lambda_{PL}(\hat a),
\end{equation}
and
\begin{equation}
\label{lyapunov-exp-eq3}
\lambda_{PL}(a^*)=\int_{D}\kappa(y-x)\frac{\phi^*(y)}{\phi^*(x)}dy+a^*(x)\quad \forall\, x\in D.
\end{equation}

By \eqref{lyapunov-exp-eq0} and \eqref{lyapunov-exp-eq3}, for any $T>0$, we have
\begin{align}
\label{lyapunov-exp-eq4}
&\lambda_{PL}(a^*)-\lambda_{PL}(a)\nonumber\\
&=\frac{1}{T}\int_0^ T \frac{v_t(t,x;u_0^*)}{v(t,x;u_0^*)}dt +\int_{D} \kappa(y-x)
\Big(\frac{\phi^*(y)}{\phi^*(x)}-\frac{1}{T}\int_0^T \frac{v(t,y;u_0^*)}{v(t,x;u_0^*)}dt\Big)dy\nonumber\\
&\,\,\, \, +a^*(x)-\frac{1}{T}\int_0^T a(t,x)dt\nonumber\\
&=\frac{1}{T}\big(\ln v(T,x;u^*_0)-\ln v(0,x;u_0^*)\big)+\int_{D}\kappa(y-x)\frac{\phi^*(y)}{\phi^*(x)}\Big(1-\frac{1}{T}\int_0^T \frac{w(t,y)}{w(t,x)}dt\Big)dy\nonumber\\
&\,\,\,\, +a^*(x)-\frac{1}{T}\int_0^T a(t,x)dt\quad \forall\, x\in D,
\end{align}
 where $w(t,x)=\frac{v(t,x;u_0^*)}{\phi_1^*(x)}$.

Choose $T>0$ such that
$$
\frac{1}{T}\int_0^T a(t,x)dt\ge \hat a(x)-\epsilon\quad \forall \, x\in D
$$
and
$$
\frac{1}{T}\big(\ln v(T,x;u_0^*)-\ln v(0,x;u_0^*)\big)= \frac{1}{T}\ln v(T,x;u_0^*)\le \frac{1}{T} \ln \|v(T,\cdot;u_0^*)\|\le \epsilon.
$$
Fix such $T$. By Lemma \ref{lyapunov-exp-lm2}, there is $x^*\in D$ such that
$$
1-\frac{1}{T}\int_0^T \frac{w(t,y)}{w(t,x^*)}dt\le 0\quad \forall\, y\in D.
$$
It then follows from \eqref{lyapunov-exp-eq2} and \eqref{lyapunov-exp-eq4} that
$$
\lambda_{PL}(\hat a)-\epsilon-\lambda_{PL}(a)\le \lambda_{PL}(a^*)-\lambda_{PL}(a)\le a^*(x)-\hat a(x)+2\epsilon\le
2\epsilon.
$$
Letting $\epsilon\to 0$, we obtain
\begin{equation}
\label{exp-estimate-eq2}
\lambda_{PL}(a)\ge \lambda_{PL}(\hat a).
\end{equation}

Next, suppose that $D$ is unbounded.
 By Theorem \ref{relation-thm2}(2) and Theorem \ref{variation-thm1}(1), we have
$$
\lambda_{PL}(a)=\lambda_{PE}(a)\ge \lambda_{PE}(\hat a)=\lambda_{PL}(\hat a).
$$
It then follows that
$$
\lambda_{PL}(a)\ge \lambda_{PL}(\hat a)\ge \sup_{x\in D}\hat a(x).
$$
where the last inequality follows from Lemma \ref{lower-bound-lem}.

(2) It follows from (1) and Theorem \ref{variation-thm1}(2).

(3) It follows from (1) and Theorem \ref{variation-thm1}(3).
\end{proof}

\section{Proof of Theorem \ref{characterization-thm1}}

In this section, we discuss the characterization of $\lambda_{PE}(a)$ and $\lambda_{PE}^{'}(a)$ and prove Theorem \ref{characterization-thm1}.

\begin{proof} [Proof of Theorem \ref{characterization-thm1}]
(1) Assume that $a(t,x)\equiv a(x)$. Let
$$
\tilde \lambda_{PE}(a)=\sup\{\lambda\,|\, \lambda\in \tilde \Lambda_{PE}(a)\}\quad {\rm and}\quad \tilde \lambda_{PE}^{'}(a)=\inf\{\lambda\,|\,
 \lambda\in\tilde \Lambda_{PE}^{'}(a)\}.
 $$

 First, by  the arguments of  Theorem \ref{relation-thm2}(1), we have
 \begin{equation}
 \label{aux-autonomous-eq1}
 \tilde\lambda_{PE}^{'}(a)=\lambda_{PL}(a).
 \end{equation}
To be  more precise,  first, when
$v(t,x)\equiv 1$ and $\lambda>\lambda_{PL}(a)$, it can be verified directly that the function $u(t,x;a,v)$ is independent of $t$, where
$u(t,x;a,v)$ is defined in \eqref{u-eq00}, that is,
$$
u(t,x;a,v)=\int_{-\infty}^ t \Phi_\lambda(t,s;a) v(s,\cdot)ds.
$$
 Then, $u(t,\cdot;a,v)\equiv u(\cdot;a,v)\in X$. By the arguments in  step 1 of the proof of Theorem \ref{relation-thm2}(1), $\lambda\in\tilde \Lambda_{PE}^{'}(a)$ and
$$
\tilde \lambda_{PE}^{'}(a)\le \lambda_{PL}(a).
$$
 Second,
it is clear that, by the arguments in step 2 of the proof of Theorem \ref{relation-thm2}(1), 
$$
\tilde\lambda_{PE}^{'}(a)\ge \lambda_{PL}(a).
$$
\eqref{aux-autonomous-eq1} thus follows.

Next, by the arguments of Theorem \ref{relation-thm2}(2), we have
\begin{equation}
\label{aux-autonomous-eq2}
\tilde\lambda_{PE}(a)=\lambda_{PL}(a).
\end{equation}
To be more precise, first, it is clear that, by the arguments in step 1 of the proof of   Theorem \ref{relation-thm2}(2),
$$
\tilde\lambda_{PE}(a)\le \lambda_{PL}(a).
$$
 Second, by the arguments in steps 2and  3 of the proof of   Theorem \ref{relation-thm2}(2), 
$$
\tilde\lambda_{PE}(a)\ge \lambda_{PL}(a).
$$
\eqref{aux-autonomous-eq2} then follows.

Now by \eqref{aux-autonomous-eq1}, \eqref{aux-autonomous-eq2}, and Theorem \ref{relation-thm2},
\begin{equation*}
\tilde\lambda_{PE}^{'}(a)=\tilde \lambda_{PE}(a)=\lambda_{PL}(a)=\lambda_{PE}(a)=\lambda_{PE}^{'}(a).
\end{equation*}
This implies (1).

 (2)  Assume that $a(t+T,x)\equiv a(t,x)$. Let
 $$
 \hat \lambda_{PE}(a)=\sup\{\lambda\,|\, \lambda\in \hat \Lambda_{PE}(a)\}\quad {\rm and}\quad \hat \lambda_{PE}^{'}(a)=\inf\{\lambda\,|\,
 \lambda\in\hat  \Lambda_{PE}^{'}(a)\}.
 $$
Similarly, by the arguments of Theorem \ref{relation-thm2}, we have
$$
\hat \lambda_{PE}^{'}(a)=\hat \lambda_{PE}(a)=\lambda_{PL}(a)=\lambda_{PE}(a)=\lambda_{PE}^{'}(a).
$$
(2) then follows.
\end{proof}


\begin{thebibliography}{99}

\bibitem{BaLi0} X. Bai and F.  Li,
 Optimization of species survival for logistic models with non-local dispersal,
  {\it Nonlinear Anal. Real World Appl.}, {\bf 21} (2015), 53-62.


\bibitem{BaZh}
P. Bates and G. Zhao, Existence, Uniqueness and Stability of the stationary solution to a nonlocal evolution equation arising in population dispersal,
{\it J. Math. Anal. Appl.}, \textbf{332}(9) (2007),  428-440.


\bibitem{BaLi} X. Bao and W.-T. Li,  Propagation phenomena for partially degenerate nonlocal dispersal models in time and space periodic habitats,
    {\it  Nonlinear Anal. Real World Appl.}, {\bf 51} (2020), 102975, 26 pp.

\bibitem{BeCoVo}
H. Berestycki, J.  Coville, and H.  Vo,  On the definition and the properties of the principal eigenvalue of some nonlocal operators,
{\it  J. of Functional Anal.},  {\bf 271} (2016),  2701-2751.

\bibitem{BeCoVo1}
H. Berestycki, J.  Coville, and H. Vo,  Persistence criteria for populations with non-local
ispersion, {\it J. Math. Biol.},  {\bf 72} (2016), 1693-1745.


\bibitem{CoElRo1} C. Cortazar, M. Elgueta, and J.D. Rossi, Nonlocal diffusion problems that approximate the heat equation with Dirichlet
boundary conditions, {\it Israel J. Math.}, {\bf  170} (2009), 53-60.

\bibitem{CoElRoWo} C. Cortazar, M. Elgueta, J.D. Rossi, and
 N. Wolanski, How to approximate the heat equation with Neumann boundary
conditions by nonlocal diffusion problems, {\it Arch. Ration. Mech. Anal.}, {\bf  187} (1) (2008), 137-156.

\bibitem{Co1}
J. Coville,
 On a simple criterion for the existence of a principal eigenfunction of some nonlocal operators,
  {\it Journal of Differential Equations}, \textbf{249} (2010), 2921-2953.


\bibitem{Co3} J. Coville, Nonlocal refuge model with a partial control,
{\it Discrete Contin. Dyn. Syst.}, {\bf  35} (4) (2015),
1421-1446.


\bibitem{CoDaMa}
J. Coville, J.  D\'avila, and S. Mart\'inez,  Existence and Uniqueness of solutions to a nonlocal equation with monostable nonlinearity, {\it  SIAM J. Math. Anal.},  \textbf{39} (2008), 1683-1709.

\bibitem{CoDaMa1} J. Coville, J. Davila, and  S. Martinez, Pulsating fronts for nonlocal dispersion and KPP nonlinearity,
{\it Ann. Inst. H. Poincaré Anal. Non Linéaire}, {\bf 30} (2013), 179-223.

\bibitem{DeShZh} P. De Leenheer, W. Shen, and A. Zhang,
 Persistence and extinction of nonlocal dispersal evolution equations in moving habitats,
  {\it Nonlinear Anal. Real World Appl.}, {\bf 54} (2020), 103110, 33 pp.

 \bibitem{Fif}P.C. Fife, An integrodifferential analog of semilinear parabolic PDEs, in: Partial Differential Equations and Applications, in: Lecture Notes in Pure and Appl. Math., vol. 177, Dekker, New York,
1996, 137-145.



\bibitem{fink}
A.M. Fink,
 Almost Periodic Differential Equations,  Lecturen Notes in Mathematics, No 377, Springer-Verlag, New York, 1974.

 \bibitem{GaRo}
J. Garcia-Melian and J.D. Rossi,
 On the principal eigenvalue of some nonlocal diffusion problems,
 {\it  J. Differ. Equ.},  \textbf{246} (2009), 21-38.

 \bibitem{GaRo1} J. Garcia-Melian, J.D. Rossi, A logistic equation with refuge and nonlocal diffusion,
 {\it Commun. Pure
Appl. Anal.}, {\bf  8} (6) (2009), 2037-2053.

 \bibitem{GrHiHuMiVi}
 M. Grinfeld, G.  Hines, V. Hutson, K. Mischaikow, and G.T. Vickers,
  Non-local Dispersal,
  {\it  Differ, Integr. Equ.},  \textbf{18} (2005), 1299-1320.

 \bibitem{Hes} P.  Hess,
Periodic-parabolic boundary value problems and positivity,
Pitman Research Notes in Mathematics Series, {\bf 247},  Longman, New York, 1991.

  \bibitem{HeShZh}
G. Hetzer, W.  Shen, and A.  Zhang,
 Effects of Spatial Variations and Dispersal Strategies on Principal eigenvalues of dispersal operators and spreading speeds of monostable equations,
 {\it  Rocky Mount. J. Math.},  \textbf{43}(2) (2013), 1147-1175.






\bibitem{HuPo} J. Huska and P.  Polacik,
 The principal Floquet bundle and exponential separation for linear parabolic equations,
  {\it J. Dynam. Differential Equations}, {\bf 16} (2004), no. 2, 347-375.

\bibitem{HuPoSa} J.  Huska, P. Polacik,  and M.V. Safonov,  Harnack inequalities, exponential separation, and perturbations of principal Floquet bundles for linear parabolic equations,
     {\it Ann. Inst. H. Poincar\'e Anal. Non Lin\'eaire}, {\bf 24} (2007), no. 5, 711-739.

\bibitem{HuMaMiVi}
V. Hutson, S.  Martinez,  K.  Mischaikow, and G.T. Vickers,
 The evolution of Dispersal, {\it J. Math. Biol.},  \textbf{47} (2003),  483-517.

\bibitem{HuShVi} V. Hutson,  W. Shen, and G.T.  Vickers,
  Spectral theory for nonlocal dispersal with periodic or almost-periodic time dependence,
{\it Rocky Mountain J. Math.},  {\bf 38} (2008), no. 4, 1147-117.


\bibitem{KaLoSh}
C.-Y. Kao, Y. Lou, and W.  Shen,
 Random Dispersal vs Non-Local Dispersal,
 {\it  Discr. Cont. Dyn. Syst.},  \textbf{26}(2) (2010), 551-596.

 \bibitem{LiCoWa} F. Li, J. Coville, and X. Wang,
  On eigenvalue problems arising from nonlocal diffusion models,
  {\it  Discrete Contin. Dyn. Syst.}, {\bf  37} (2017), no. 2, 879-903.

\bibitem{LiSuWa}
W.-T. Li, Y.-J.  Sun, and Z.-C. Wang,
 Entire Solutions in the Fisher-KPP equation with Nonlocal Dispersal,
  {\it Nonlinear Analysis: Real World Appl.},  \textbf{11}(4) (2010), 2302-2313.

  \bibitem{LiWaZh} W.-T. Li, J.-B. Wang, and X.-Q. Zhao,
  Spatial dynamics of a nonlocal dispersal population model in a shifting environment,
   {\it J. Nonlinear Sci.}, {\bf  28} (2018), no. 4, 1189-1219.

 \bibitem{LiZh} X.  Liang and T. Zhou,
  Spreading speeds of nonlocal KPP equations in almost periodic media,
  {\it  J. Funct. Anal.}, {\bf 279} (2020), no. 9, 108723, 58 pp.

  \bibitem{LuPaLe} F. Lutscher, E. Pachepsky,  and M.A. Lewis, The effect of dispersal patterns on stream populations,
{\it SIAM Rev.}, {\bf 47} (4) (2005), 749-772.

\bibitem{MiSh1} J. Mierczyński and W.  Shen,  Exponential separation and principal Lyapunov exponent/spectrum for random/nonautonomous parabolic equations,
    {\it  J. Differential Equations}, {\bf 191} (2003), no. 1, 175-205.

 \bibitem{MiSh2} J. Mierczyński and W. Shen, Spectral theory for random and nonautonomous parabolic equations and applications, Chapman \& Hall/CRC Monographs and Surveys in Pure and Applied Mathematics, {\bf 139},  CRC Press, Boca Raton, FL, 2008.


\bibitem{NadRos}
G. Nadin and L.  Rossi,
 Propagation phenomena for time heterogeneous KPP reaction-diffusion equations,
 {\it  J. Math. Pures Appl.},  {\bf (9) 98} (2012), no. 6, 633-653.

\bibitem{Paz}
A.L. Pazy,  Semigroups of Linear Operators and Applications to Partial Differential Equations. Springer, New York, 1983.

\bibitem{RaSh}
N. Rawal and W.  Shen,
 Criteria for the existence and lower bounds of principal eigenvalues of time periodic nonlocal dispersal operators and applications,
  {\it J. Dynam. Differential Equations}, {\bf  24} (2012), 927-954.


\bibitem{RaShZh}
N. Rawal, W.  Shen, and A.  Zhang,
 Spreading speeds and traveling waves of nonlocal monostable equations in time and space periodic habitats,
  {\it Discrete Contin. Dyn. Syst.},  {\bf 35} (2015), no. 4, 1609-1640.

\bibitem{ShVi} W. Shen and  G.T. Vickers, Spectral theory for general nonautonomous/random dispersal evolution operators,
{\it J. Differential
Equations}, {\bf 235} (1) (2007), 262-297.

\bibitem{ShXi} W. Shen and X.  Xie,  Spectral theory for nonlocal dispersal operators with time periodic indefinite weight functions and applications,
    {\it  Discrete Contin. Dyn. Syst. Ser. B}, {\bf  22} (2017), no. 3, 1023-1047.

\bibitem{ShXi1}
W. Shen and X.   Xie,
 On principal spectrum points/principal eigenvalues of nonlocal dispersal operators and applications,
  {\it Discrete and Continuous Dynamical Systems}, \textbf{35} (2015), 1665-1696.

\bibitem{ShXi2}
W. Shen and X.  Xie,
Approximations of random dispersal operators/equations by nonlocal dispersal operators/equations,
 {\it J. Differential Equations}, {\bf  259} (2015), no. 12, 7375-7405.


\bibitem{ShYi} W. Shen and Y.  Yi,
Almost automorphic and almost periodic dynamics in skew-product semiflows,  Mem. Amer. Math. Soc.,  {\bf  136} (1998), no. 647.



\bibitem{ShZh1}
W. Shen and A. Zhang,
Spreading speeds for monostable equations with nonlocal dispersal in space periodic habitats,
{\it Journal of Differential Equations}, \textbf{249} (2010), 747-795.

 \bibitem{ShZh2}
W.  Shen and A. Zhang,
 Stationary solutions and spreading speeds of nonlocal monostable equations in space periodic habitats, {\it  Proc. Amer. Math. Soc.}, {\bf 140} (2012), no. 5, 1681-1696.

\bibitem{ShVo}
Z. Shen and H.-H. Vo,
 Nonlocal dispersal equations in time-periodic media: principal spectral theory, limiting properties and long-time dynamics,
  {\it J. Differential Equations}, {\bf  267} (2019), no. 2, 1423-1466.

  \bibitem{SuLiLoYa} Y.-H. Su, W.-T. Li, Y. Lou, and F.-Y. Yang,
   The generalised principal eigenvalue of time-periodic nonlocal dispersal operators and applications,
    {\it J. Differential Equations}, {\bf 269} (2020), no. 6, 4960-4997.

 \bibitem{Tur}  P. Turchin, Quantitative Analysis of Movement: Measuring and Modeling Population Redistribution
in Animals and Plants, Sinauer Associates, 1998.


\bibitem{ZhZh1} G.-B. Zhang and X.-Q. Zhao,
  Propagation dynamics of a nonlocal dispersal Fisher-KPP equation in a time-periodic shifting habitat,
   {\it J. Differential Equations}, {\bf 268} (2020), no. 6, 2852-2885.

 \bibitem{ZhZh2} G.-B.  Zhang and X.-Q. Zhao,  Propagation phenomena for a two-species Lotka-Volterra strong competition system with nonlocal dispersal,
     {\it  Calc. Var. Partial Differential Equations}, {\bf  59} (2020), no. 1, Paper No. 10, 34 pp.

\end{thebibliography}
\end{document}